\documentclass{amsart}
\usepackage{amsthm}
\usepackage{amsmath}
\usepackage{amssymb}
\usepackage{comment}
\usepackage{enumerate}
\usepackage{amsfonts}
\usepackage{mathtools}
\usepackage{afterpage}
\usepackage{hyperref}
\usepackage{todonotes}
\usepackage{comment}
\usepackage{algorithm,algorithmic}
\usepackage{multicol}
\usepackage{tikz-cd}
\usepackage[utf8]{inputenc}

\newtheorem{theorem}{Theorem}[section]
\newtheorem{def-prop}[theorem]{Definition-Proposition}

\newtheorem{lemma}[theorem]{Lemma}

\theoremstyle{definition}
\newtheorem{ex}[theorem]{Example}
\newtheorem{quest}[theorem]{Question}
\newtheorem{defin}[theorem]{Definition}

\theoremstyle{remark}
\newtheorem*{remark}{Remark}

\def\D{\mathsf{D}}
\def\E{\mathsf{E}}
\def\pop{\operatorname{pop}}
\def\wt{\operatorname{wt}}

\def\perm{\operatorname{perm}}
\def\DES{\operatorname{Des}}
\def\BPD{\operatorname{BPD}}
\def\PD{\operatorname{PD}}

\def\blank{\operatorname{blank}}
\def\cross{\operatorname{cross}}
\def\Red{\operatorname{Red}}
\def\+{\includegraphics[scale=0.4]{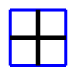}}
\def\bl{\includegraphics[scale=0.4]{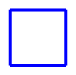}}
\def\bt{\includegraphics[scale=0.4]{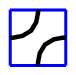}}
\def\rt{\includegraphics[scale=0.4]{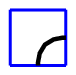}}
\def\jt{\includegraphics[scale=0.4]{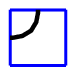}}
\def\elbow{\includegraphics[scale=0.4]{bump.eps}}
\def\vtile{\includegraphics[scale=0.4]{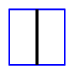}}
\def\htile{\includegraphics[scale=0.4]{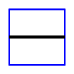}}
\def\mindroop{\textsf{min-droop}}
\def\cbswap{\textsf{cross-bump-swap}}
\renewcommand{\S}{\mathfrak{S}}

\newcommand{\Z}{\mathbb{Z}}

\title[The canonical bijection between PDs and BPDs]{The canonical bijection between pipe dreams and bumpless pipe dreams}
\author{Yibo Gao}
\author{Daoji Huang}

\date{\today}

\begin{document}

\begin{abstract}
We present a direct bijection between reduced pipe dreams and reduced bumpless pipe dreams by interpreting reduced compatible sequences on bumpless pipe dreams and show that this bijection preserves Monk's formula, establishing its canonical nature.
\end{abstract}
\maketitle

\section{Introduction}
Schubert polynomials, $\{\S_\pi\:|\: \pi\in S_n\}$, defined by Lascoux and Sch\"utzenberger \cite{lascoux1982polynomes}, are polynomial representatives of the Schubert classes of the full flag variety $\mathrm{Fl}_n$. They can be defined via the divided difference operators as follows:
\[\S_\pi:=
\begin{cases}
x_1^{n-1}x_2^{n-2}\cdots x_{n-1} & \text{ if }w=n\ n{-}1\ \cdots\ 1,\\
\partial_i \S_{\pi s_i} & \text{ if } \ell(\pi)<\ell(\pi s_i),
\end{cases}
\]
where $s_i$ is the simple transposition $(i\ i{+}1)$, which acts on polynomials by swapping the variables $x_i$ and $x_{i+1}$, and 
\[\partial_i f:=\frac{f-s_if}{x_i-x_{i+1}}.\]

Schubert polynomials and their various generalizations are the central objects of interest in Schubert calculus and algebraic combinatorics, and possess very rich combinatorial structures. The Schubert polynomials $\S_\pi$'s expand positively into monomials. Notable combinatorial models for this result involve \textit{compatible sequences} by Billey-Jockusch-Stanley \cite{BJS}, which are equivalent to \textit{pipe dreams} (PDs), also called \textit{RC-graphs}, by Bergeron-Billey \cite{bergeron1993rcgraph}.
\begin{theorem}[\cite{bergeron1993rcgraph, BJS}]\label{thm:Schubert-PD}
Let $\pi$ be a permutation, then $\S_\pi=\sum_{D\in\PD(\pi)}\wt(D).$
\end{theorem}
We will introduce the necessary notations in Section~\ref{sec:preliminaries}.

Relatively recently, Lam, Lee and Shimozono \cite{LLS} introduced \textit{bumpless pipe dreams} (BPDs), in the context of back stable Schubert calculus, that provide another monomial expansion of the Schubert polynomials.
\begin{theorem}[\cite{LLS}]\label{thm:Schubert_BPD}
Let $\pi$ be a permutation, then $\S_\pi=\sum_{D\in\BPD(\pi)}\wt(D).$
\end{theorem}
Theorem~\ref{thm:Schubert-PD} and Theorem~\ref{thm:Schubert_BPD} guarantee that there exists a weight-preserving bijection between pipe dreams and bumpless pipe dreams of a fixed permutation. The following question has been of interest since the discovery of bumpless pipe dreams.
\begin{quest}\label{que:main}
Can we explicitly describe a ``natural" weight-preserving bijection between $\PD(\pi)$ and $\BPD(\pi)$ for any permutation $\pi$?
\end{quest}
The goal of this paper is to answer the above question affirmatively, by providing such a bijection $\varphi$ (Definition~\ref{def:bijection}). We then show that our bijection is the ``correct" bijection in a precise sense by proving that $\varphi$ preserves Monk's rule (Theorem~\ref{thm:preserve-monk}). 

One notable consequence of Theorem~\ref{thm:preserve-monk} is that it unifies all attempts to define bijections between pipe dreams and bumpless pipe dreams via any choices of Monk's rules. Specifically, the second author mentioned a family of bijections \cite{huang2020bijective} between pipe dreams and bumpless pipe dreams defined using (co)transition formulas which are special cases of Monk's rule. Each bijection in this family depends on a choice $\alpha$ for each permutation $\pi$ upon which the (co)transition is performed. By Theorem~\ref{thm:preserve-monk}, we now know that this family of bijections is in fact one single bijection (Definition~\ref{def:bijection}). 

For the organization of this paper, we first introduce background knowledge in Section~\ref{sec:preliminaries}. Then we describe the bijection in Section~\ref{sec:bijection}. Finally, we show that the bijection preserves Monk's rule in Section~\ref{sec:Monk}, which is divided into many subsections due to the technicality the proof of our main theorem, Theorem~\ref{thm:preserve-monk}.

\section{Preliminaries}\label{sec:preliminaries}
Let $S_{\infty}$ be the infinite symmetric group that consists of bijections from $\Z_{>0}$ to itself with all but finitely many fixed points. Namely, $S_\infty=\bigcup_{n\ge 1} S_n$. The group $S_{\infty}$ is generated by the simple transpositions $s_i:=(i\ i{+}1)$, for $i\in\Z_{>0}$. For $a<b$, let $t_{a,b}$ be the transposition $(a\ b)$. Let $\ell(w)$ be the standard Coxeter length of $w$ and let $\Red(w)$ be the set of reduced words of $w$. 

The Bruhat order on $S_{\infty}$ is the transitive closure of the binary relation $w<wt_{a,b}$ if $\ell(w)<\ell(wt_{a,b})$. The cover relation in the Bruhat order is given by $x\lessdot y$ if $x=yt_{a,b}$ for some $a<b$ and $\ell(x)=\ell(y)-1$. Also, let $\DES_L(\pi):=\{i\in \Z_{>0}\:|\: s_i \pi\lessdot \pi\}$ be the set of (left) descents of $\pi$. 

The stability property of the Schubert polynomials allows us to work with permutations $\pi\in S_{\infty}$, instead of each symmetric group $S_n$ separately. Recall that the stability property states that for any $m<n$, $\pi\in S_m$, we have $\S_\pi=\S_{\iota(\pi)}$, where $\iota:S_m\hookrightarrow S_n$ is the natural embedding.

\subsection{Pipe dreams and compatible sequences}
\begin{defin}[Pipe dream, RC-graph \cite{bergeron1993rcgraph}]\label{def:pipe-dream}
For a permutation $\pi\in S_{\infty}$, a (reduced) \textit{pipe dream} $D$ of $\pi$ is a tiling of the square grid $\Z_{>0}\times\Z_{>0}$ using two kinds of tiles, the cross-tile \+ and the elbow-tile \elbow, with finitely many \+-tiles, forming pipes that travel from the north border to the west border, such that for $i\in\Z_{>0}$, the pipe starting at column $i$, which is labeled as pipe $i$, ends at row $\pi(i)$, and  no two pipes cross twice. 
\end{defin}

We typically draw a pipe dream in a staircase shape of size $n$ if $\pi\in S_n$. Let $\PD(\pi)$ be the set of pipe dreams of $\pi$, and for $D\in\PD(\pi)$, let $\cross(D)$ denote the coordinates of its \+-tiles. Then the \textit{weight} of $D$ is defined as \[\wt(D):=\prod_{(i,j)\in\cross(D)}x_i.\] 
Given a pipe dream $D$, we let $\perm(D)$ denote the permutation given by $D$.
\begin{defin}[Compatible sequence \cite{BJS}]\label{def:compatible-sequence}
For a permutation $\pi\in S_\infty$ with $\ell(\pi)=\ell$, a tuple of integer sequences $\big(\mathbf{a}=(a_1,\ldots,a_{\ell}),\mathbf{r}=(r_1,\ldots,r_{\ell})\big)$ is a (reduced) \textit{compatible sequence} of $\pi$ if
\begin{enumerate}
\item $\mathbf{a}=(a_1,\ldots,a_{\ell})$ is a reduced word of $\pi$,
\item $r_1\leq\cdots \leq r_{\ell}$ is weakly increasing,
\item $r_j\leq a_j$ for $j=1,\ldots,\ell$,
\item $r_j<r_{j+1}$ if $a_j<a_{j+1}$.
\end{enumerate}
\end{defin}

There is a straightforward bijection between pipe dreams of $\pi$ and compatible sequences of $\pi$, which we describe here.

We label the square grids with matrix notation and give a total order to the square grids, by going from top to bottom, and within each row, right to left. To be precise, $(i,j)<(i',j')$ if $i<i'$ or $i=i',j>j'$. For $D\in\PD(\pi)$ with $\ell(\pi)=\ell$, order its \+-tiles $\cross(D)$ in this way as $(r_1,j_1)<\cdots<(r_{\ell},j_{\ell})$. Then its corresponding compatible sequence is $\big(\mathbf{a}=(a_1,\ldots,a_{\ell}),\mathbf{r}=(r_1,\ldots,r_{\ell})\big)$ where $a_k=r_k+j_k-1$. Note that each \+-tile at coordinate $(i,j)$ corresponds to the simple transposition $s_{i+j-1}$. See \cite{bergeron1993rcgraph} for further details on this bijection. 

\begin{ex}\label{ex:pd}
Figure~\ref{fig:pd-ex1} shows two pipe dreams of $\pi=21543$ of weight $x_1^2x_2x_3$ and their corresponding compatible sequences.
\begin{figure}[h!]
\centering
\includegraphics[scale=0.4]{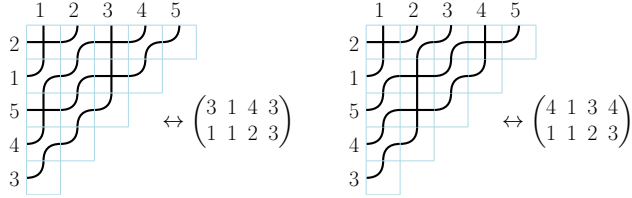}
\caption{Examples of pipe dreams and their corresponding compatible sequences}
\label{fig:pd-ex1}
\end{figure}
\end{ex}

As the bijection between pipe dreams and compatible sequences is quite straightforward, we will abuse notation and from now on, consider them to be the same combinatorial objects for simplicity. 

For a pipe dream $D\in\PD(\pi)$ with compatible sequence $(\mathbf{a},\mathbf{r})$, its \textit{first} \+-tile corresponding to the order on the square grid will play an important role later on, which has coordinate $(r_1,a_1-r_1+1)$. We denote by $\nabla D$ the pipe dream obtained from $D$ by turning this first \+-tile into a \elbow-tile, and write $\pop(D)=(a_1,r_1)$. We have that $\nabla D\in\PD(s_{a_1}\pi)$, where $a_1\in\DES_L(\pi)$ so that $\ell(s_{a_1}\pi)=\ell(\pi)-1$. 

\subsection{Bumpless pipe dreams}
\begin{defin}[Bumpless pipe dream \cite{LLS}]
For a permutation $\pi\in S_{\infty}$, a (reduced) \textit{bumpless pipe dream} $D$ of $\pi$ is a tiling of the square grid $\Z_{>0}\times\Z_{>0}$ using the following six tiles,
\begin{center}
\includegraphics[scale=0.5]{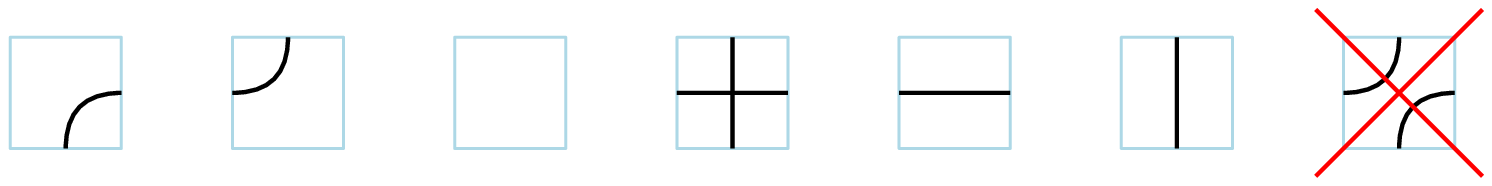}
\end{center}
with finitely many \+-tiles, forming pipes labeled by $\Z_{>0}$ such that pipe $i$ travels from $(\infty,i)$ to $(\pi(i),\infty)$ in the NE direction and that no two pipes cross twice. Among the tiles, ``\rt'' is pronounced ``are'' and ``\jt'' is pronounced ``jay''. The name ``bumpless'' comes from the fact that the  tile ``\bt'' where two pipes ``bump'' is forbidden.
\end{defin}
We typically draw a bumpless pipe dream in a square grid of size $n\times n$ if $\pi\in S_n$. Let $\BPD(\pi)$ denote the set of bumpless pipe dreams of $\pi$, and for $D\in\BPD(\pi)$, let $\cross(D)$ denote the coordinates of its \+-tiles and $\blank(D)$ denote the coordinates of its \bl-tiles. Then the \textit{weight} of $D$ is defined as \[\wt(D):=\prod_{(i,j)\in\blank(D)}x_i.\] 

\section{Description of the bijection}\label{sec:bijection}
We start with the following operation on bumpless pipe dreams, which is defined in the rectification process described in \cite{huang2021schubert}. These moves are generalizations of the backwards direction of the insertion proceess described in \cite[Section 5]{LLS}. We follow the definition of \emph{droop} and \emph{undroop} moves on bumpless pipe dreams as in \cite[Section 5]{LLS}.
\begin{defin}\label{def:pop}
Given $D\in\BPD(\pi)$ with $\ell(\pi)>0$, the following process produces another bumpless pipe dream $\nabla D\in\BPD(\pi')$ where $\ell(\pi')=\ell(\pi)-1$. Suppose $D$ is an $n\times n$ grid. Let $r$ be the smallest row index such that the row $r$ of $D$ contains \bl-tiles. To initialize, mark the rightmost \bl-tile in row $r$ with a label ``$\times$". 
\begin{enumerate}
\item If the marked \bl-tile is not the rightmost \bl-tile in a contiguous block of \bl-tiles in its row, move the label ``$\times$" to the rightmost \bl-tile of this block. Assume the marked tile has coordinate $(x,y)$, and the pipe going through $(x,y+1)$ is $p$. 
\item If $p\neq y+1$, suppose the \jt-tile of $p$ in column $y+1$ has coordinate $(x',y+1)$ for some $x'>x$. Call the rectangle with NW corner $(x,y)$ and SE corner $(x',y+1)$ the \emph{column move rectangle} $U$. We modify the tiles in $U$ as follows:

\begin{enumerate}

    \item For each pipe $q$ intersecting $p$ at some $(z,y+1)$ where $x<z<x'$ and $(z,y)$ is an \rt-tile, let $(z',y)$ be the \jt-tile of $q$ in column $y$. Ignoring the presence of $p$, droop $q$ at $(z,y)$ within $U$, so that $(z,y+1)$ becomes an \rt-tile and $(z',y+1)$ becomes a \jt-tile;
    \item Undroop pipe $p$ at $(x',y+1)$ into $(x,y)$, and move the mark to $(x',y+1)$.
\end{enumerate}
Go back to step (1) and repeat. The column moves are illustrated in Figure \ref{fig:columnmoves}.
\begin{figure}[h!]
\centering
\includegraphics[scale=0.3]{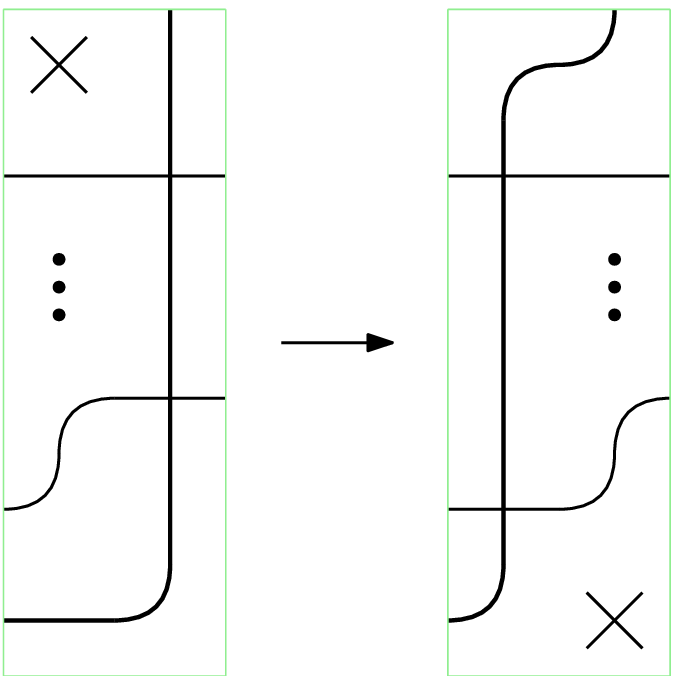}
\qquad
\includegraphics[scale=0.3]{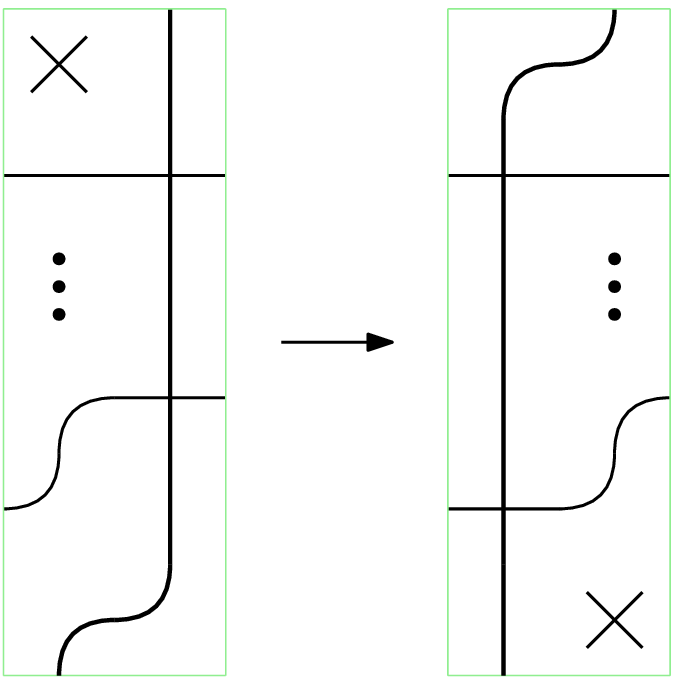}
\qquad
\includegraphics[scale=0.3]{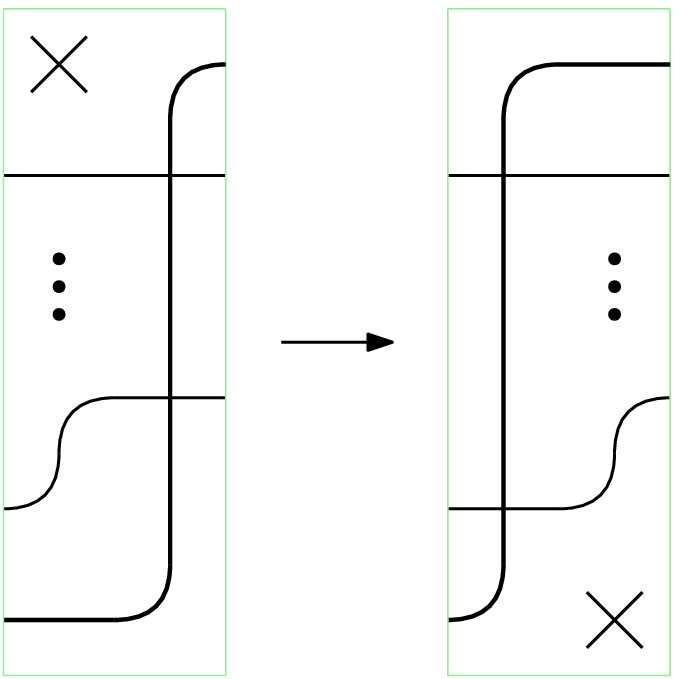}
\qquad
\includegraphics[scale=0.3]{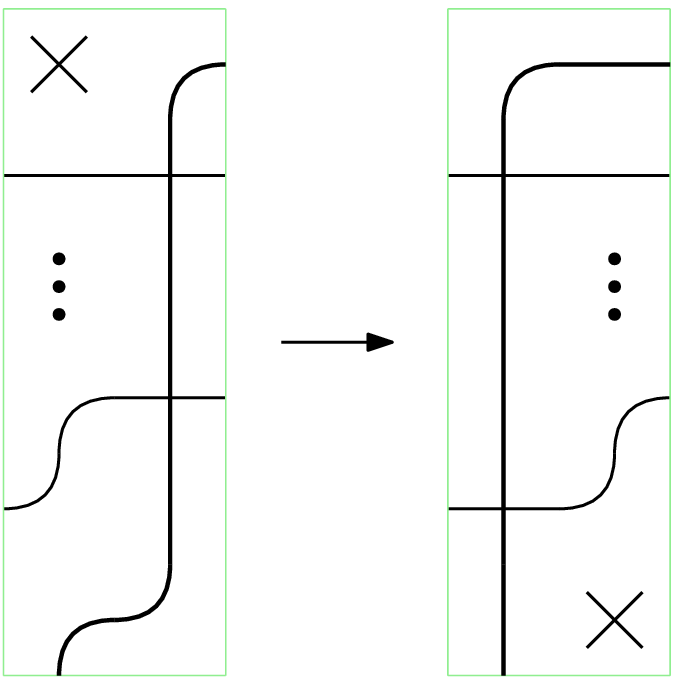}
\caption{Column moves}
\label{fig:columnmoves}
\end{figure}
\item 
If $p=y+1$, the pipes $y$ and $y+1$ must intersect at some $(x',y+1)$ for some $x'>x$. Replace  this \+-tile with a \bt-tile, undroop the \jt-turn of this tile into $(x,y)$ and adjust the pipes between row $x$ and $x'$ so that their ``kinks shift right'', in a same fashion as described in  Step (2) above. In this case, call the rectangle with NW corner $(x,y)$ and SE corner $(n,y+1)$ the \emph{column move rectangle.} These moves are shown in Figure~\ref{fig:deletemark}.
We are done after this step.

\begin{figure}[h!]
\centering
\includegraphics[scale=0.3]{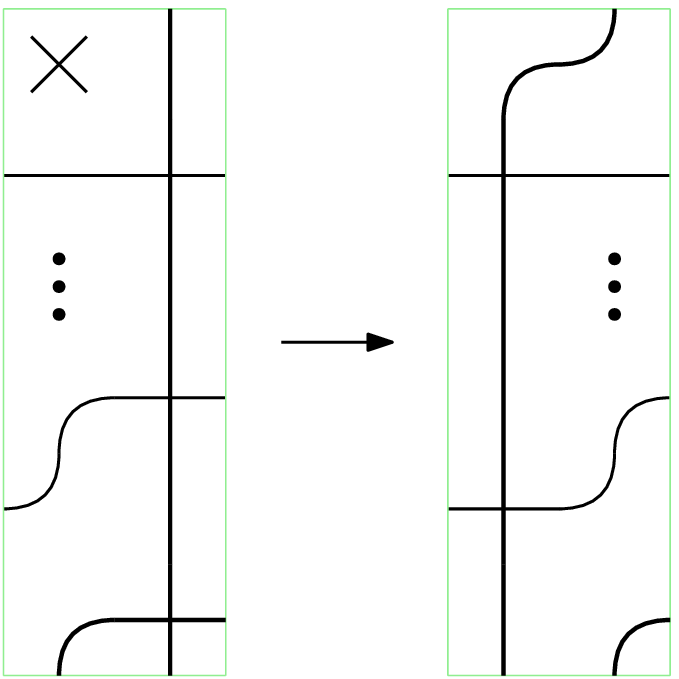}
\qquad\qquad
\includegraphics[scale=0.3]{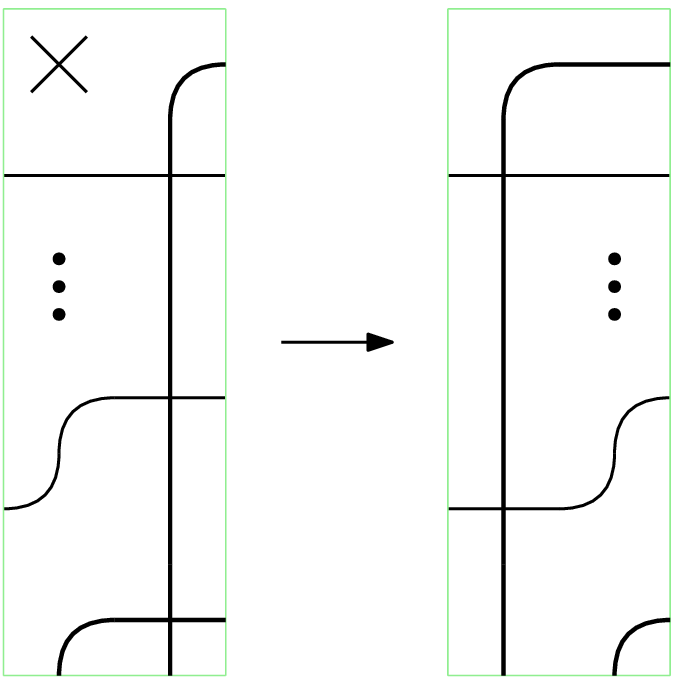}
\caption{The final step of deleting a marked \bl-tile}
\label{fig:deletemark}
\end{figure}
\end{enumerate}
Let $a$ be the column index of the left column of the last column move rectangle as in Step (3).
The final BPD is of the permutation $s_a\pi$. Denote the result by $\nabla D\in\BPD(s_a\pi)$, and write $\pop(D)=(a, r)$. Let the \textit{footprints} of $D$ be the set of coordinates that are SE corners of the column move rectangles except for the southernmost one.
\end{defin}
Note that each step above is invertible. Specifically, given $D'$ and a pair $(a,r)$, we can uniquely recover $D$, if it exists, such that $\nabla D=D'$ and $\pop(D)=(a,r)$, by inverting the above steps: start with crossing the pipe $a$ and pipe $a+1$ where the pipe $a+1$ first turn right (if pipe $a$ and $a+1$ already cross in $D'$, then we know that such $D$ does not exist) and creating a \bl-tile in some row $r'$; then keep doing backward-direction column moves and sliding the \bl-tile to the left until the \bl-tile reaches row $r$. If this \bl-tile cannot land exactly on row $r$, we also conclude that such $D$ does not exist.

\begin{ex}
Figure~\ref{fig:pop-ex1} demonstrates how Definition~\ref{def:pop} works for a certain bumpless pipe dream $D\in\BPD(\pi)$ step by step, where $\nabla D\in\BPD(s_a\pi)$ is obtained in the end, with $\pi=2153746$, $a=4$, $\pop(D)=(4,1)$.
\begin{figure}[h!]
\centering
\includegraphics[scale=0.28]{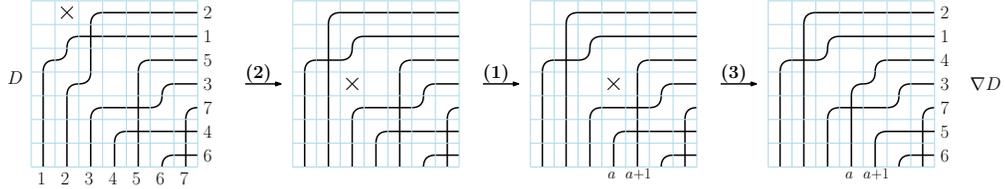}
\caption{An example of obtaining $\nabla D$ with the step number labeled on the arrows}
\label{fig:pop-ex1}
\end{figure}
\end{ex}

Repeatedly applying the procedure in Definition~\ref{def:pop}, we obtain the the following map, which is the main object of study of this paper.
\begin{defin}[The bijection]\label{def:bijection}
Given $D\in\BPD(\pi)$ with $\ell(\pi)=\ell$, let \[\varphi(D)=\big(\mathbf{a}=(a_1,\ldots,a_{\ell}),\mathbf{r}=(r_1,\ldots,r_{\ell})\big),\] where $\pop(\nabla^{i-1}D)=(a_i,r_i)$ for $i=1,\ldots,\ell$.
\end{defin}
\begin{ex}
Figure~\ref{fig:bijection-21543-ex1} and Figure~\ref{fig:bijection-21543-ex2} demonstrate two examples of the bijection $\varphi$ in Definition~\ref{def:bijection} for $\pi=21543$, whose Schubert polynomial $\mathfrak{S}_{\pi}$ is not multiplicity-free in the monomial expansion (see \cite{fink2021zero-one}). Both examples  have weight $x_1^2x_2x_3$.
\begin{figure}[h!]
\centering
\includegraphics[scale=0.3]{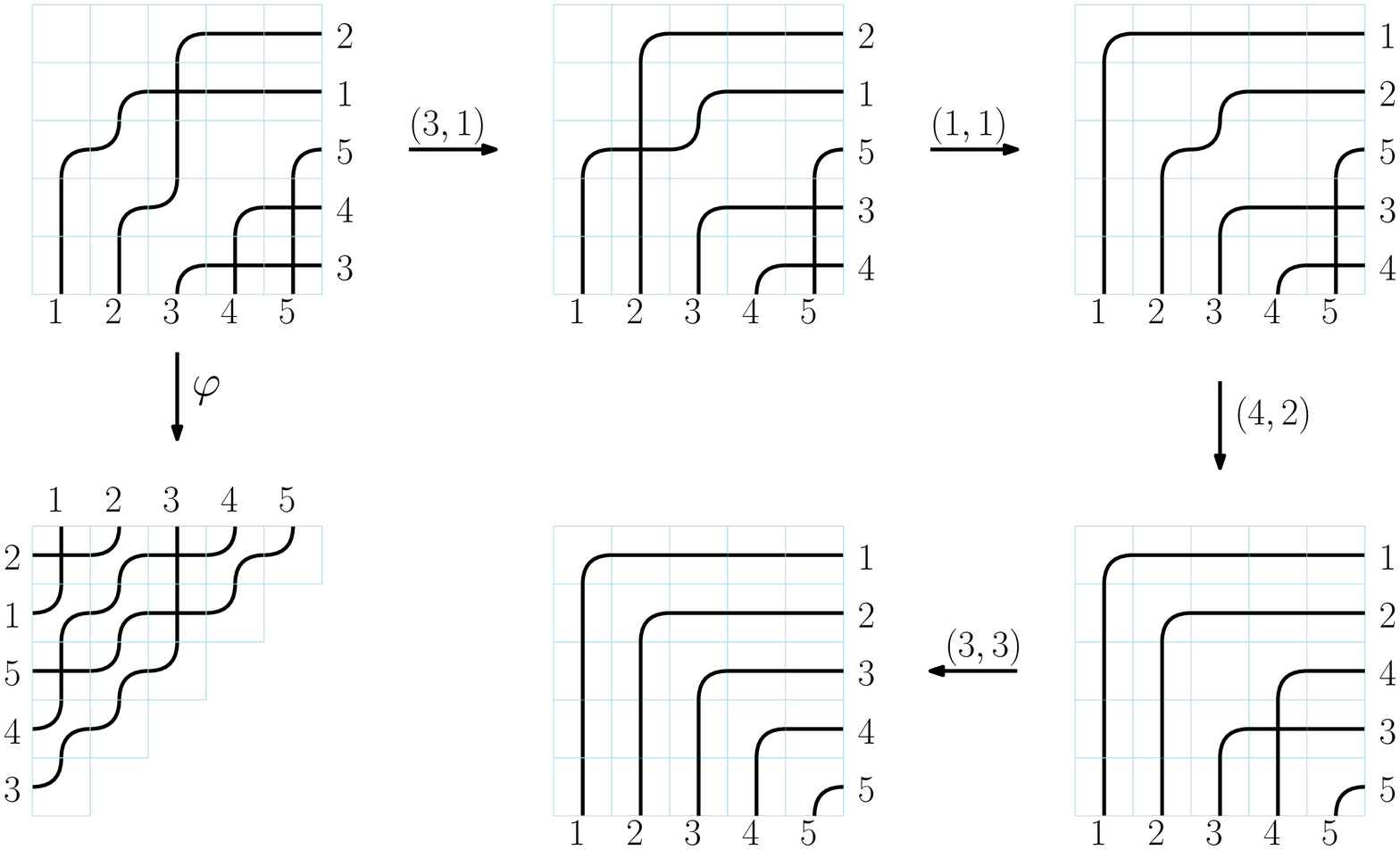}
\caption{An example for the bijection $\varphi$ with $\pi=21543$, where the values of $\pop$'s are shown on the arrows.}
\label{fig:bijection-21543-ex1}
\end{figure}
\begin{figure}[h!]
\centering
\includegraphics[scale=0.3]{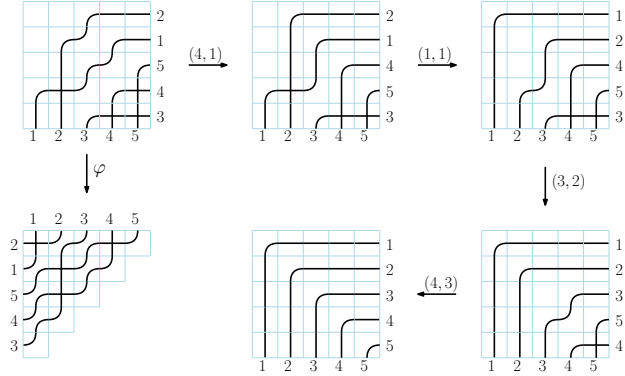}
\caption{Another example for the bijection $\varphi$ with $\pi=21543$, where the values of $\pop$'s are shown on the arrows.}
\label{fig:bijection-21543-ex2}
\end{figure}
\end{ex}

\begin{lemma}\label{lem:bijection-welldefined}
For $D\in\BPD(\pi)$, $\varphi(D)$ is a compatible sequence of $\pi$.
\end{lemma}
\begin{proof}
We check the conditions of compatible sequences for $\varphi(D)$ as in Definition~\ref{def:compatible-sequence}. Condition (1) follows from $\nabla D\in\BPD(s_a\pi)$ if $\pop(D)=(a,r)$. Condition (2) is evident from construction, as we are removing \bl-tiles from top to bottom. Note that if $\pop(D)=(a,r)$, with $D\in\BPD(\pi)$, there are no \bl-tiles in the first $r-1$ rows of $D$, meaning that $\pi(i)=i$ for $i=1,\ldots,r-1$, and consequently $a\geq r$. Thus, condition (3) follows. For condition (4), we assume the opposite that $r_{j}\geq r_{j+1}$, i.e. $r_j=r_{j+1}$ and $a_j<a_{j+1}$. Since $a_j<a_{j+1}$, by Lemma 3.2 of \cite{huang2021schubert}, the footprints of $\nabla^{j-1}D$ are strictly to the S/W/SW direction of the footprints of $\nabla^j D$. However, the starting point of the footprints of $\nabla^{j-1}D$ is directly east of the footprints of $\nabla^{j}D$ as we start in the same row $r_j=r_{j+1}$. Contradiction. 
\end{proof}
\begin{theorem}\label{thm:bijection}
The map $\varphi$ in Definition~\ref{def:bijection} is a weight-preserving bijection between $\BPD(\pi)$ and $\PD(\pi)$, i.e. compatible sequences of $\pi$. 
\end{theorem}
\begin{proof}
By Lemma~\ref{lem:bijection-welldefined}, we have a well-defined map $\varphi$ from $\BPD(\pi)$ to $\PD(\pi)$. Immediate from the construction (Definition~\ref{def:bijection}), $\varphi$ is weight-preserving. Since each step $\nabla$ is invertible if a preimage exists, the map $\varphi$ is injective. As both pipe dreams and bumpless pipe dreams enumerate the Schubert polynomial $\mathfrak{S}_\pi$, $\varphi$ is a bijection. 
\end{proof}

\section{Preserving Monk's rule}\label{sec:Monk}
In this section, we demonstrate why the map $\varphi$ in Definition~\ref{def:bijection} is ``canonical", by showing that it preserves Monk's rule. 
\begin{theorem}[Monk's rule]\label{thm:monk}
Let $\pi$ be a permutation and $\alpha$ be a positive integer,
\begin{equation}\label{eq:Monk}
x_{\alpha}\S_\pi+\sum_{\substack{s<\alpha\\\pi t_{s,\alpha}\gtrdot\pi}}\S_{\pi t_{s,\alpha}}=\sum_{\substack{l>\alpha\\\pi t_{\alpha,l}\gtrdot\pi}}\S_{\pi t_{\alpha,l}}.
\end{equation}
\end{theorem}

This version of Monk's rule is derived from the Monk's rule that expands the product of a linear Schubert polynomial with a general Schubert polynomial. In the cases when the summation on the right-hand of Equation~\eqref{eq:Monk} is a single Schubert polynomial, the formula specializes to the \emph{transition formula} of Schubert polynomials, which is an important inductive formula for Schubert polynomials. In the case when the summation on the left-hand side is empty, the formula specializes to the \emph{cotransition formula} \cite{knutson2019schubert}. In \cite{billey2019bijective}, a bijective proof of the transition formula is given using pipe dreams, which can be easily extended to a proof of Monk's rule in its general form. The second author gives a bijective proof of Monk's rule with bumpless pipe dreams in \cite{huang2020bijective}, which also generalizes to the equivariant version. It is also remarked in \cite{huang2020bijective} that using either the transition or the cotransition formula, one can construct  bijections inductively of pipe dreams and bumpless pipe dreams. As a corollary of our result, these  inductive bijections agree. We review both weight-preserving bijections for Equation~\eqref{eq:Monk} on pipe dreams and bumpless pipe dreams in Section~\ref{sub:monk-pd} and Section~\ref{sub:monk-bpd}.

\subsection{Monk's rule on pipe dreams}\label{sub:monk-pd}
We present the following maps
\[
\begin{cases}
x_\alpha{\rightsquigarrow}:&\PD(\pi)\rightarrow \bigcup_{\substack{l>\alpha\\\pi t_{\alpha,l}\gtrdot\pi}}\PD(\pi t_{\alpha,l})\\
m_{s,\beta}:&\PD(\pi t_{s,\beta})\rightarrow \bigcup_{\substack{l>\beta\\\pi t_{\beta,l}\gtrdot\pi}}\PD(\pi t_{\beta,l}),\ s<\beta,\pi t_{s,\beta}\gtrdot\pi
\end{cases}
\]
such that $x_{\alpha}{\rightsquigarrow}$ and $m_{s,\beta}$'s for $\beta=\alpha$ together form a weight-preserving bijection from the left-hand side to the right-hand side of Equation~\eqref{eq:Monk}. To be precise on the weight, $x_{\alpha}{\rightsquigarrow}$ multiplies the weight by $x_{\alpha}$ while $m_{s,\beta}$ preserves the weight. Readers are referred to \cite{billey2019bijective}  for further details, and how these maps can be inverted. 

\begin{defin}[$x_{\alpha}{\rightsquigarrow},m_{s,\beta}$ on  pipe dreams]\label{def:monk-pd}
Given $D\in\PD(\pi)$, the following procedure produces $x_{\alpha}{\rightsquigarrow} D$.
\begin{enumerate}
\item 
Find the leftmost \elbow-tile on row $\alpha$ of $D$ and replace it by \+. 
\item If the newly added \+-tile creates a double crossing with another \+-tile at coordinate $(i,j)$, replace the \+-tile at $(i,j)$ with \elbow, find the smallest $j'>j$ such that $(i,j')$ is a \elbow-tile, and replace it by \+. Repeat this step.
\end{enumerate}
Analogously, given $\pi$, $D\in\PD(\pi t_{s,\beta})$ such that $\pi t_{s,\beta}\gtrdot\pi$, the following procedure produces $m_{s,\beta}(D)\in\PD(\pi t_{\beta,l})$ for some $l>\beta$ where $\pi t_{\beta,l}\gtrdot \pi$. 
\begin{enumerate}
\item Locate the \+-tile between pipe $\pi^{-1}(s)$ and $\pi^{-1}(\beta)$ in $D$ at coordinate $(i,j)$. Make it into a \elbow-tile. Find the smallest $j'>j$ such that $(i,j')$ is a \elbow-tile, and replace it by \+.
\item Exactly the same as step (2) above.
\end{enumerate}
\end{defin}
See Figure~\ref{fig:monk-pd-ex1} for an example.

\subsection{Monk's rule on bumpless pipe dreams}\label{sub:monk-bpd}
We present the analogous maps
\[
\begin{cases}
x_\alpha{\rightsquigarrow}:&\BPD(\pi)\rightarrow \bigcup_{\substack{l>\alpha\\\pi t_{\alpha,l}\gtrdot\pi}}\BPD(\pi t_{\alpha,l})\\
m_{s,\beta}:&\BPD(\pi t_{s,\beta})\rightarrow \bigcup_{\substack{l>\beta\\\pi t_{\beta,l}\gtrdot\pi}}\BPD(\pi t_{\beta,l}),\ s<\beta,\pi t_{s,\beta}\gtrdot\pi
\end{cases}
\]
for bumpless pipe dreams. For details including illustrated examples, see \cite{huang2020bijective}. First we recall from \cite{huang2020bijective} that an \textit{almost bumpless pipe dream} of $\pi$ is defined by relaxing the condition of bumpless pipe dreams by allowing a single \bt-tile in the grid.

\begin{defin}[Basic Monk moves on (almost) bumpless pipe dreams]\label{def:basic-monk-moves}
We define two \emph{basic Monk moves}, $\mindroop$ and $\cbswap$, on (almost) bumpless pipe dreams. See Figure \ref{fig:basic-monk-moves}.
\vskip 0.5em

\noindent $\mindroop$: Let $(a,b)$ be the position of an \rt-turn of a pipe $p$. Note that the tile at $(a,b)$ could be a \rt-tile or \bt-tile. Let $x>0$ be the smallest number where $(a+x,b)$ is not a \+-tile, and $y>0$ be the smallest number where $(a,b+y)$ is not a \+-tile. A $\mindroop$ at $(a,b)$ droops $p$ into $(a+x,b+y)$.
\vskip 0.5em

\noindent $\cbswap$: Suppose $(a,b)$ is a $\bt$-tile of pipes $p$ and $q$, and $p$ and $q$ also have a crossing at $(a',b')$. Then a $\cbswap$ move at $(a,b)$ swaps the two tiles at $(a,b)$ and $(a',b')$.

Furthermore, we consider the initial move that replaces a \+-tile with a \bt-tile when computing $m_{s,\beta}$, and the final move that replaces a \bt-tile with a \+-tile also as special basic Monk moves.

\begin{figure}[h!]
    \centering
    \includegraphics[scale=0.7]{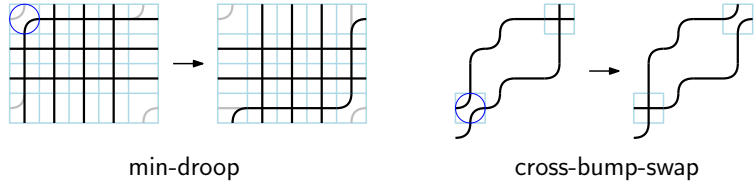}
    \caption{Basic monk moves at the circled coordinates}
    \label{fig:basic-monk-moves}
\end{figure}
\end{defin}

By the stability property of Schubert polynomials, we may embed $\pi$ into a larger symmetric group as necessary so that a $\mindroop$ is always possible.

\begin{defin}[$x_{\alpha}{\rightsquigarrow},m_{s,\beta}$ on bumpless pipe dreams]\label{def:monk-bpd}
Given $D\in\BPD(\pi)$, the following procedure produces $x_\alpha \rightsquigarrow D$.
\begin{enumerate}
    \item Let $(\alpha, j)$ be the easternmost \rt-tile in row $\alpha$, initialize $(\mathsf{x},\mathsf{y}):=(\alpha,j)$.
    \item Perform a $\mindroop$ at $(\mathsf{x},\mathsf{y})$. Let $(z,w)$ be the SE corner of this $\mindroop$.
    \begin{enumerate}
        \item If after the $\mindroop$, $(z,w)$ is a \jt-tile, update $(\mathsf{x},\mathsf{y})$ as the position of \rt-tile of pipe $\pi(\alpha)$ in row $z$, and repeat this step.
        
        \item If after this $\mindroop$, $(z,w)$ is a \bt-tile, let $q$ be the pipe of the \rt-turn in this tile. 
        \begin{enumerate}
            \item If $\pi(\alpha)$ and $q$ have already intersected at $(z',w')$, perform a $\cbswap$ move at $(z,w)$. Update $(\mathsf{x},\mathsf{y}):=(z',w')$, and repeat Step (2). 
            \item If $\pi(\alpha)$ and $q$ have not intersected before, replace the \bt-tile with a \+-tile and stop.
        \end{enumerate}
    \end{enumerate}
\end{enumerate}
To define $m_{s,\beta}(D)$, replace Step (1) with
\begin{enumerate}
    \item[(1')] Initialize $(\mathsf{x},\mathsf{y})$ to be the position of the \+-tile of $\pi(s)$ and $\pi(\beta)$. Replace this tile with a \bt-tile.
\end{enumerate}

\end{defin}

\subsection{The main theorem and the proof}
The following is the main theorem of the paper.
\begin{theorem}\label{thm:preserve-monk}
For any $\pi\in S_{\infty}$, the bijection $\varphi:\BPD(\pi)\rightarrow\PD(\pi)$ preserves Monk's rule, i.e. $\varphi$ intertwines with $x_{\alpha}{\rightsquigarrow}$, $m_{s,\beta}$'s. Specifically, the following diagrams commute:
\begin{center}
\begin{tikzcd}
\BPD(\pi)\arrow[r,"x_{\alpha}{\rightsquigarrow}"]\arrow[d,"\varphi"] & \bigcup_{\substack{l>\alpha\\\pi t_{\alpha,l}\gtrdot\pi}}\BPD(\pi t_{\alpha,l})\arrow[d,"\varphi"] \\
\PD(\pi)\arrow[r,"x_{\alpha}{\rightsquigarrow}"] & \bigcup_{\substack{l>\alpha\\\pi t_{\alpha,l}\gtrdot\pi}}\PD(\pi t_{\alpha,l})
\end{tikzcd},
\begin{tikzcd}
\BPD(\pi t_{s,\beta})\arrow[r,"m_{s,\beta}"]\arrow[d,"\varphi"] & \bigcup_{\substack{l>\beta\\\pi t_{\beta,l}\gtrdot\pi}}\BPD(\pi t_{\beta,l})\arrow[d,"\varphi"] \\
\PD(\pi t_{s,\beta})\arrow[r,"m_{s,\beta}"] & \bigcup_{\substack{l>\beta\\\pi t_{\beta,l}\gtrdot\pi}}\PD(\pi t_{\beta,l})
\end{tikzcd},
\end{center}
where $\alpha\in\Z_{>0}$ is arbitrary, and $\pi t_{s,\beta}\gtrdot \pi$ with $s<\beta$.
\end{theorem}

Despite how simple Theorem~\ref{thm:preserve-monk} is stated, the proof turns out to rely on  technical lemmas on pipe dreams (Lemma~\ref{lem:PD-induction-main}) and bumpless pipe dreams (Lemma~\ref{lem:BPD-induction-main}), which  discuss how much the operator $\nabla$ commutes with $x_{\alpha}{\rightsquigarrow}$'s and $ m_{s,\beta}$'s respectively on pipe dreams and bumpless pipe dreams. We finish the main proof here assuming that both lemmas are already taken care of.
\begin{proof}[Proof of Theorem~\ref{thm:preserve-monk}]
This theorem is a direct consequence of the induction principle and the fact that $\nabla$ and $x_{\alpha}{\rightsquigarrow}, m_{s,\beta}$ interact in the exact same way on pipe dreams (Lemma~\ref{lem:PD-induction-main}) and bumpless pipe dreams (Lemma~\ref{lem:BPD-induction-main}).

To be precise, we proceed by induction on $\ell(\pi)$. The base case is $\ell(\pi)=0$, i.e. $\pi=\mathrm{id}$. Take $B\in\BPD(\pi)$, which has no crossings. Its corresponding pipe dream $D:=\varphi(B)$ also has no crossings. By the definition of $x_{\alpha}{\rightsquigarrow}$, $x_{\alpha}{\rightsquigarrow} B$ has a crossing between pipe $\alpha$ and $\alpha+1$ at coordinate $(\alpha+1,\alpha+1)$ and the insertion creates a \bl-tile on row $\alpha$, while $x_{\alpha}{\rightsquigarrow} D$ has a single crossing at coordinate $(\alpha,1)$. By Definition~\ref{def:pop}, $\pop(x_\alpha{\rightsquigarrow} B)=(\alpha,\alpha)$ which gives us $\varphi(x_\alpha{\rightsquigarrow} B)=x_{\alpha}{\rightsquigarrow} D$. For the second commutative diagram, in order for $\pi t_{s,\beta}\gtrdot\pi$, we must have $s=\beta-1$. Any $B\in\BPD(\pi t_{s,\beta})$ must have a crossing between pipe $\beta-1$ and $\beta$ at coordinate $(\beta,\beta)$ and a \bl-tile at coordinate $(k,k)$ for some $k<\beta$. The operation $m_{s,\beta}$ uncrosses the \+-tile between pipe $\beta-1$ and $\beta$ and inserts a \+-tile between pipe $\beta$ and $\beta+1$ at coordinate $(\beta+1,\beta+1)$, keeping the position of the \bl-tile. Thus, we see that $\pop(B)=(\beta-1,k)$ and $\pop(m_{s,\beta}(B))=(\beta,k)$, meaning that $\varphi(B)$ has a single \+-tile at coordinate $(k,\beta-k)$ and $\varphi(m_{s,\beta}(B))$ has a single \+-tile at coordinate $(k,\beta-k+1)$ so evidently $m_{s,\beta}(\varphi(B))=\varphi(m_{s,\beta}(B))$ as desired. See Figure~\ref{fig:preserve-monk-base-case}.
\begin{figure}[h!]
\centering
\includegraphics[scale=0.3]{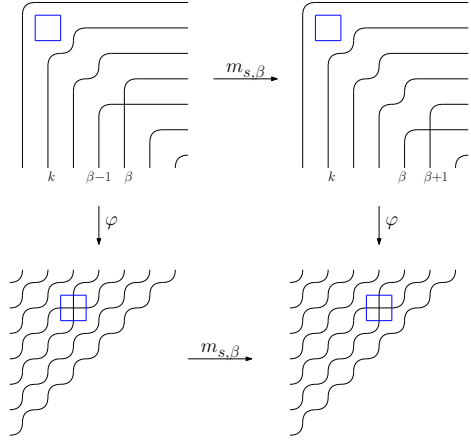}
\caption{The base case for the second commutative diagram of Theorem~\ref{thm:preserve-monk}}
\label{fig:preserve-monk-base-case}
\end{figure}

Now assume that we have established the commutative property of the above diagrams for permutations $\pi'$ with $\ell(\pi')<\ell$. Now fix $\pi$ with $\ell(\pi)=\ell$ so that $\ell(\pi t_{s,\beta})=\ell+1$. We argue about the second commutative diagram first. Take $B\in\BPD(\pi t_{s,\beta})$ and let $D=\varphi(B)\in\PD(\pi t_{s,\beta})$. Since $D$ and $B$ are bijection, let $\pop(B)=\pop(D)=(i,r)$. Our goal is to show that the bumpless pipe dream $m_{s,\beta}(B)$ and the pipe dream $m_{s,\beta}(D)$ are in bijection via $\varphi$. By Definition~\ref{def:bijection} of $\varphi$, it suffices to show that $\pop (m_{s,\beta}(B))=\pop(m_{s,\beta}(D))$ and that $\nabla(m_{s,\beta}(B))$ and $\nabla(m_{s,\beta}(D))$ are in bijection via $\varphi$.

\vskip 0.5em
\noindent\textbf{Case (1)(a)} of Lemma~\ref{lem:PD-induction-main} and Lemma~\ref{lem:BPD-induction-main}: $(s,\beta)\neq(\pi^{-1}(i+1),\pi^{-1}(i))$. By induction hypothesis where $\ell(\perm(\nabla D))=\ell(\perm(\nabla B))=\ell(\pi t_{s,\beta})-1=\ell$, as $\nabla B$ and $\nabla D$ are in bijection, $m_{s,\beta}(\nabla B)$ and $m_{s,\beta}(\nabla D)$ must be in bijection. Therefore, we can let $\rho:=\perm(m_{s,\beta}(\nabla B))=\perm(m_{s,\beta}(\nabla D))$. We further divide into subcases based on whether $i\in\DES_L(\rho)$ and $i+1\in\DES_L(\rho)$, as indicated by Lemma~\ref{lem:PD-induction-main} and Lemma~\ref{lem:BPD-induction-main}. 

If $i\notin\DES_L(\rho)$, by Lemma~\ref{lem:PD-induction-main} and Lemma~\ref{lem:BPD-induction-main}, $\pop(m_{s,\beta}(D))=(i,r)=\pop(m_{s,\beta}(B))$. Furthermore, as $m_{s,\beta}(\nabla D)$ and $m_{s,\beta}(\nabla B)$ are in bijection via $\varphi$, $\nabla(m_{s,\beta}(D))=m_{s,\beta}(\nabla D)$ and $\nabla(m_{s,\beta}(B))=m_{s,\beta}(\nabla B)$ are in bijection.

If $i\in\DES_L(\rho)$ and $i+1\notin$, $\pop(m_{s,\beta}(D))=(i+1,r)=\pop(m_{s,\beta}(B))$ by Lemma~\ref{lem:PD-induction-main} and Lemma~\ref{lem:BPD-induction-main}. Then as above, $\nabla(m_{s,\beta}(D))=m_{s,\beta}(\nabla D)$ and $\nabla(m_{s,\beta}(B))=m_{s,\beta}(\nabla B)$ are in bijection.

If $i,i+1\in\DES_L(\rho)$, $\pop(m_{s,\beta}(D))=(i+1,r)=\pop(m_{s,\beta}(B))$. As $m_{s,\beta}(\nabla D)$ and $m_{s,\beta}(\nabla B)$ are in bijection with $\ell(\rho)<\ell(\pi t_{s,\beta})$, induction hypothesis gives us that $m_{\rho^{-1}(i+2),\rho^{-1}(i+1)}(m_{s,\beta}(\nabla D))$ and $m_{\rho^{-1}(i+2),\rho^{-1}(i+1)}(m_{s,\beta}(\nabla B))$ are in bijection via $\varphi$. So $\nabla(m_{s,\beta}(D))$ and $\nabla(m_{s,\beta}(B))$ are in bijection as desired.

\vskip 0.5em
\noindent\textbf{Case (1)(b)} of Lemma~\ref{lem:PD-induction-main} and Lemma~\ref{lem:BPD-induction-main}: $(s,\beta)=(\pi^{-1}(i+1),\pi^{-1}(i))$. The same argument as in Case (1)(a) works here, by going through the different scenarios of whether $i,i+1\in\DES_L(\rho)$ and reading off that $\pop(m_{s,\beta}(B))=\pop(m_{s,\beta}(D))$ and that $\nabla(m_{s,\beta}(B))$ and $\nabla(m_{s,\beta}(D))$ are in bijection via induction hypothesis. 

\vskip 0.5em
The first commutative diagram, i.e. Case (2) of Lemma~\ref{lem:PD-induction-main} and Lemma~\ref{lem:BPD-induction-main}, follows basically from the same argument.

\vskip 0.5em
\noindent\textbf{Case (2)(a)} of Lemma~\ref{lem:PD-induction-main} and Lemma~\ref{lem:BPD-induction-main}: $\alpha\geq r$. As $\nabla D$ and $\nabla B$ are in bijection, by induction hypothesis, the PD $x_{\alpha}{\rightsquigarrow}(\nabla D)$ and the BPD $x_{\alpha}{\rightsquigarrow}(\nabla B)$ are in bijection via $\varphi$. So we can let $\rho=\perm(x_{\alpha}{\rightsquigarrow}(\nabla D))=\perm(x_{\alpha}{\rightsquigarrow}(\nabla B))$. As above, we see that $\pop(x_{\alpha}{\rightsquigarrow}D)=(i,r)=\pop(x_{\alpha}{\rightsquigarrow}B)$ if $i\notin\DES_L(\rho)$ and $\pop(x_{\alpha}{\rightsquigarrow}D)=(i+1,r)=\pop(x_{\alpha}{\rightsquigarrow}B)$ if $i\in\DES_L(\rho)$, so $\pop(x_{\alpha}{\rightsquigarrow}D)=\pop(x_{\alpha}{\rightsquigarrow}B)$.

Moreover, if $i,i+1\in\DES_L(\rho)$, as $x_{\alpha}{\rightsquigarrow}(\nabla D)$ and $x_{\alpha}{\rightsquigarrow}(\nabla B)$ are in bijection with $\ell(\rho)\leq\ell$, by  Case (1), $m_{\rho^{-1}(i+2),\rho^{-1}(i+1)}(x_{\alpha}{\rightsquigarrow}(\nabla D))=\nabla(x_{\alpha}{\rightsquigarrow}D)$ and $m_{\rho^{-1}(i+2),\rho^{-1}(i+1)}(x_{\alpha}{\rightsquigarrow}(\nabla B))=\nabla(x_{\alpha}{\rightsquigarrow}B)$ must be in bijection. If not both of $i,i+1$ belong in $\DES_L(\rho)$, we also have that $\nabla(x_{\alpha}{\rightsquigarrow}D)=x_{\alpha}{\rightsquigarrow}(\nabla D)$ is in bijection with $\nabla(x_{\alpha}{\rightsquigarrow}B)=x_{\alpha}{\rightsquigarrow}(\nabla B)$ by induction hypothesis.

Case (2)(b) is done in the same way.
\end{proof}
\subsection{Induction on pipe dreams}
Referring back to Definition~\ref{def:monk-pd}, we introduce more notations for $x_{\alpha}{\rightsquigarrow}$ and $m_{s,\beta}$'s on pipe dreams. Fix a permutation $\pi$ as in Section~\ref{sub:monk-pd}. These maps consist of \textit{basic Monk steps} of deleting \+-tiles (i.e. turning a \+-tile into a \elbow-tile), which we denote as $k^-$, and steps of adding \+-tiles (i.e. turning a \elbow-tile into a \+-tile), which we denote as $k^+$, $k=1,2,\ldots$. Write $(i_{k^-},j_{k^-})$ for the coordinate where step $k^-$ happens, and $(i_{k^+},j_{k^+})$ for the coordinate where step $k^+$ happens. To be precise, the map $x_{\alpha}{\rightsquigarrow}$ consists of steps $1^+,2^-,2^+,3^-,\ldots,q^+$ for some $q\geq1$ and the map $m_{s,\beta}$ consists of steps $1^-,1^+,2^-,2^+,\ldots,p^-,p^+$ for some $p\geq1$. We say that the sequence of coordinates $(i_{1^+},j_{1^+}),(i_{2^-},j_{2^-}),\ldots,(i_{q^+},j_{q^+})$ is the \textit{Monk footprints} for $x_{\alpha}{\rightsquigarrow}$ and the sequence $(i_{1^-},j_{1^-}),(i_{1^+},j_{1^+}),\ldots,(i_{p^+},j_{p^+})$ is the \textit{Monk footprints} for $m_{s,\beta}$.

Note that after a step $k^-$, we have a genuinely reduced pipe dream for $\pi$, and after each step $k^+$ except the last one, we have a double crossing at coordinate $(i_{k^+},j_{k^+})$ and $(i_{k+1^{-}},j_{k+1^{-}})$. Also, by definition, $i_k^{-}=i_k^{+}$ for all $k$. Let $i_k:=i_k^{-}=i_k^{+}$.
Also define, for $m_{s,\beta}$, the \textit{complete Monk footprints} to be
\[\bigcup_{k=1}^p\{(i_k,j_{k^-}),(i_k,j_{k^-}+1),\ldots,(i_k,j_{k^+})\},\]
and for $x_{\alpha}{\rightsquigarrow}$, the \textit{complete Monk footprints} to be
\[\{(i_1,j_{1^+})\}\bigcup_{k=2}^q\{(i_k,j_{k^-}),(i_k,j_{k^-}+1),\ldots,(i_k,j_{k^+})\}.\]
In other words, the complete Monk footprints are the Monk footprints union all the coordinates strictly between $(i_k,j_{k^-})$ and $(i_k,j_{k^+})$, whose tiles are all \+'s throughout the process of Definition~\ref{def:monk-pd}, as we show momentarily. 
\begin{ex}
Figure~\ref{fig:monk-pd-ex1} shows an example of a step by step computation of the Monk's rule on a pipe dream $D\in\PD(\pi t_{s,\beta})$ where $\pi=21786534$, $(s,\beta)=(2,5)$. 
\begin{figure}[h!]
\centering
\includegraphics[scale=0.3]{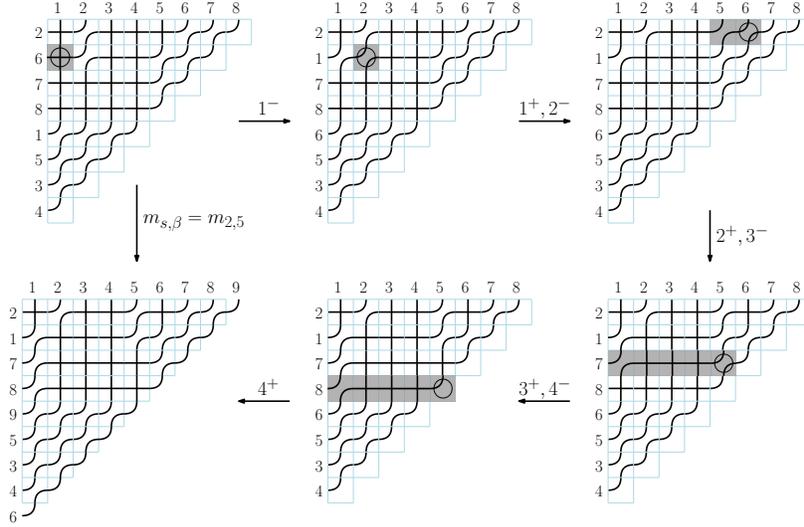}
\caption{An example for basic Monk steps for $m_{s,\beta}$ on some $D\in\PD(\pi t_{s,\beta})$ where $\pi=21786534$, $(s,\beta)=(2,5)$. The tiles that are going to be modified immediately are circled, and the complete Monk footprints are shaded. Notice that area above pipe $6$ is growing towards the SE direction.}
\label{fig:monk-pd-ex1}
\end{figure}
\end{ex}

\begin{lemma}\label{lem:monk-increase-one}
The complete Monk footprints consist of distinct coordinates.
\end{lemma}
\begin{proof}
We first consider $m_{s,\beta}$ applied to $D\in\PD(\pi t_{s,\beta})$. Keep the notations above for Monk footprints $(i_{1^-},j_{1^-}),(i_{1^+},j_{1^+}),\ldots,(i_{p^+},j_{p^+})$. Let $b=\pi(\beta)$, which is the larger pipe number among the two pipes that intersect at $(i_{1^-},j_{1^-})$ in $D$. 

For $k=1,\ldots,p$, let $D^{k}\in\PD(\pi)$ be the pipe dream obtained from $D$ after the basic Monk step $k^-$. The readers are strongly recommended to refer to Figure~\ref{fig:monk-pd-ex1} for a visualization, where $b=\pi^{-1}(5)=6$. 

We use induction on $k$ to show that the grid $(i_{k^+}, j_{k^+})$ in $D^k$, which is a \elbow-tile by definition, contains the pipe $b$ as its \jt\  part.

For the base case $k=1$, starting with $D$, we replace the \+-tile at $(i_{1^-},j_{1^-})$ with a \elbow-tile, so the \rt\ part of this tile is pipe $b$. By Definition~\ref{def:monk-pd}, we search towards the right in $D^1$ for the first \elbow-tile at $(i_{1^+}, j_{1^+})$, which must contain pipe $b$ as its \jt\ part. 

Assume that we have proved this claim for $k-1$, $k\leq p$. To do the step $(k-1)^+$, we insert a \+-tile in $D^{k-1}$ at coordinate $(i_{k-1^+},j_{k-1^+})$, which is a \elbow-tile in $D^{k-1}$ that contains pipe $b$ as its \jt\ part by induction hypothesis. As this creates a double crossing with the \+-tile in $D^{k-1}$ at $(i_{k^-},j_{k^-})$, there must be two pipes $b$ and $c$ that pass through the tiles $(i_{k-1^+},j_{k-1^+})$ and $(i_{k^-},j_{k^-})$ in $D^{k-1}$; and moreover, pipe $b$ must be on the NW side of pipe $c$ within this region, since it starts as the \jt\ part of $(i_{k-1^+},j_{k-1^+})$. Denote these two paths between $(i_{k-1^+},j_{k-1^+})$ and $(i_{k^-},j_{k^-})$ as $pb$ and $pc$, for pipe $b$ and pipe $c$ respectively. After we insert a \+-tile at $(i_{k-1^+},j_{k-1^+})$ and delete the \+-tile at $(i_{k^-},j_{k^-})$ to obtain $D^k$, pipe $b$ now travels between these two tiles via path $pc$, and contains the \rt\ part of $(i_{k^-},j_{k^-})$. Therefore, searching from this tile towards the right for the first \elbow-tile at $(i_{k^+},j_{k^+})$, we see that it's \jt\ part belongs to pipe $b$. The induction step is now complete.

The claim and the above analysis immediately implies that the area bounded by the pipe $b$ and the $x,y$-axis in $D^k$ strictly increases as $k$ increases. This monotonicity is enough for us to conclude the lemma statement, as the added tiles in the complete Monk footprints after each step lie between the pipe $b$ in $D^{k-1}$ and the pipe $b$ in $D^k$, so they will never overlap.

The proof for $x_{\alpha}{\rightsquigarrow}$ is exactly the same so we will not repeat the details here. 
\end{proof}

We note that Lemma~\ref{lem:monk-increase-one} gives us that $\pop(m_{s,\beta}(D))$, $\pop(x_{\alpha}{\rightsquigarrow} D)$ is either $(i,r)$ or $(i+1,r)$, if $\pop(D)=(i,r)$. 

\begin{lemma}\label{lem:PD-induction-main}
Let $\pi\in S_\infty$. The following statements are true.
\begin{enumerate}
\item Suppose $D\in\PD(\pi t_{s,\beta})$, where $s<\beta$ and $\pi t_{s,\beta}\gtrdot \pi$. Let $\pop(D)=(i,r)$.

\begin{enumerate}
\item If $(s,\beta)\neq (\pi^{-1}(i+1), \pi^{-1}(i))$, let  $\rho:=\perm(m_{s,\beta} (\nabla D))$. Then 
\[\pop(m_{s,\beta}(D))=
\begin{cases}(i+1,r) & \text{ if } i\in \DES_L(\rho)\\(i,r) & \text {otherwise}
\end{cases}. \] 
Furthermore, 
\[\nabla(m_{s,\beta}(D))=
\begin{cases}m_{\rho^{-1}(i+2),\rho^{-1}(i+1)}(m_{s,\beta}(\nabla D)) & \text{ if } i,i+1\in \DES_L(\rho)\\
m_{s,\beta}(\nabla D) & \text {otherwise}
\end{cases}. \] 
\item If $(s,\beta)= (\pi^{-1}(i+1), \pi^{-1}(i))$, then $\pop(m_{s,\beta}(D))=(i+1,r)$. Furthermore, let $\rho:=\perm(\nabla(D))$. Then,
\[\nabla(m_{s,\beta}(D))=
\begin{cases}m_{\rho^{-1}(i+2),\rho^{-1}(i+1)}(\nabla D) & \text{ if } i+1\in \DES_L(\rho)\\
\nabla D & \text {otherwise}
\end{cases}. \] 
\end{enumerate}

\item Suppose $D\in\PD(\pi)$ and $\pop(D)=(i,r)$. 
\begin{enumerate}
\item If $\alpha\ge r$, let $\rho:=\perm(x_\alpha{\rightsquigarrow} (\nabla D))$. Then \[\pop(x_\alpha{\rightsquigarrow} D)=
\begin{cases}(i+1,r) & \text{ if } i\in \DES_L(\rho)\\(i,r) & \text {otherwise}
\end{cases}. \] 
Furthermore, 
\[\nabla(x_\alpha{\rightsquigarrow} D)=
\begin{cases}
m_{\rho^{-1}(i+2),\rho^{-1}(i+1)}(x_\alpha{\rightsquigarrow}(\nabla D)) & \text{ if } i,i+1\in \DES_L(\rho)\\
x_\alpha{\rightsquigarrow}(\nabla D) & \text {otherwise}
\end{cases}. \] 
\item If $\alpha < r$, then $\pop(x_\alpha \rightsquigarrow D)=(\alpha, \alpha)$, and $\nabla(x_\alpha{\rightsquigarrow} D)=D$.
\end{enumerate}
\end{enumerate}
\end{lemma}

Figure~\ref{fig:pd-subcase1a3} demonstrates an example of a critical case of this Lemma, showing how the operators $\nabla$ and $m_{s,\beta}$'s on the pipe dreams interact. 

\begin{proof}[Proof of Lemma~\ref{lem:PD-induction-main}]
We will focus more on Case (1)(a), since Case (1)(b) is essentially a degenerate Case of (1)(a), and Case (2) uses the same argument as Case (1).
\vskip 0.5em

\noindent\textbf{Case (1)(a):} $D\in\PD(\pi t_{s,\beta})$, $s<\beta$, $\pi t_{s,\beta}\gtrdot\pi$, $\pop(D)=(i,r)$ and $(s,\beta)\neq(\pi^{-1}(i+1),\pi^{-1}(i))$. Recall that $\pop(D)=(i,r)$ means that the first \+-tile of $D$ is located at $(r,i-r+1)$, i.e. the $r$th row is the top row of $D$ with a \+-tile, and that the \+-tile in the $r$th row furthest to the right has coordinate $(r,i-r+1)$, which is the intersection of pipe $i$ and $i+1$. The pipe dream $\nabla D$ is obtained from $D$ by deleting this \+-tile.

Assume that $m_{s,\beta}D$ consists of basic Monk steps $1^-,1^+,\ldots,p^-,p^+$, at Monk footprints $F=\{(i_{1^-},j_{1^-}),\ldots,(i_{p^+},j_{p^+})\}$. Let the complete Monk footprints be \[C=\bigcup_{k=1}^p C_k,\quad\text{where }C_k:= \{(i_{k},j_{k^-}),(i_k,j_{k^-}+1),\ldots,(i_k,j_{k^+})\}.\]
We further divide into subcases. In all the subcases, let $\rho:=\perm(m_{s,\beta}(\nabla D))$.
\vskip 0.5em

\noindent\textbf{Subcase (1)(a)(i):} $(r,i-r+1)\notin C$. 

This condition of the current subcase directly gives us $\pop(m_{s,\beta}(D))=(i,r)$. We then compare the process of $m_{s,\beta}$ applied to $D$ and $\nabla D$. Each basic monk move that is applied to $D$ will be applied to $\nabla D$ in the exact same way, as the first \+-tile of $D$, which is the only difference between $D$ and $\nabla D$, is not part of the conversation. In the end, we obtain that $m_{s,\beta}(\nabla D)$ equals $m_{s,\beta}(D)$ taken away the first \+-tile at $(r,i-r+1)$, i.e. $m_{s,\beta}(\nabla D)=\nabla(m_{s,\beta}(D))$. Moreover, recall that $\rho=\perm(m_{s,\beta}(\nabla D))$, so $\perm(m_{s,\beta}D)=s_i\rho$ since $m_{s,\beta}D$ and $m_{s,\beta}(\nabla D)$ differs at the \+-tile at $(r,i-r+1)$ which corresponds to the simple transposition $s_i$. By the reduced criterion of these pipe dreams, $i\notin\DES_L(\rho)$. 
\vskip 0.5em

\noindent\textbf{Subcase (1)(a)(ii):} $(r,i-r+1)\in C_p$.

This condition of the current subcase is saying that when we are doing the Monk's rule $m_{s,\beta}$D, the last two basic steps are deleting the \+-tile at $(i_p,j_{p^-})$, which resolves a previous double crossing, and then inserting a \+-tile at $(r,i-r+2)$, which does not create a double crossing so that we stop here. Consequently, $\pop(m_{s,\beta}D)=(i+1,r)$ by definition. 

Again, we compare the process $m_{s,\beta}(D)$ with $m_{s,\beta}(\nabla D)$. The basic Monk steps are the same before step $p$. At step $(p-1)^+$, in $m_{s,\beta}(D)$, a double crossing is created with the \+-tile at coordinate $(i_p,j_{p^-})=(r,j_{p^-})$, which is one of the \+-tile in $D$ belonging to the block of rightmost \+-tiles in row $r$. We split the discussions based on whether $(r,j_{p^-})$ is the rightmost \+-tile in row $r$ of $D$, i.e. whether $(r,j_{p^-})$ is in $\nabla D$ or not, and then arrive at the same conclusion.

If $(r,j_{p^-})\notin\cross(\nabla D)$, i.e. $j_{p^-}=i-r+1$, we know that $m_{s,\beta}(\nabla D)$ ends at step $(p-1)^+$. The permutation $\rho=\perm(m_{s,\beta}(\nabla D))$ is also the permutation we obtained from the process of $m_{s,\beta}$ applying to $D$ after move $p^-$ so $\rho=\pi$ in fact. The move $p^-$ removes a \+-tile that corresponds to the simple transposition $s_i$, so $i\in\DES_L(\rho)$. And the move $p^+$ adds a \+-tile at coordinate $(r,i-r+2)$ that corresponds to the simple transposition $s_{i+1}$, so $\perm(m_{s,\beta}D)=s_{i+1}\rho$ and $i+1\notin\DES_L(\rho)$. Moreover, $\nabla(m_{s,\beta}D)$ removes the \+-tile at $(r,i-r+2)$ from $m_{s,\beta}D$ so we observe that $\nabla(m_{s,\beta}D)=m_{s,\beta}(\nabla D)$.

If $(r,j_{p^-})\in\cross(\nabla D)$, i.e. $j_{p^-}<i-r+1$, then the process of $m_{s,\beta}$ applied to $\nabla D$ deletes the \+-tile at $(r,j_{p^-})$ at step $p^-$, and adds a \+-tile at $(r,i-r+2)$ at step $p^+$, arriving at the pipe dream that equals $m_{s,\beta}$ applying to $D$ after step $p^-$, which is reduced so we stop. This also gives $\nabla(m_{s,\beta}D)=m_{s,\beta}(\nabla D)$. In $m_{s,\beta}(\nabla D)$, the coordinate $(r,i-r+1)$ contains a \+-tile so $i\in\DES_L(\rho)$ and since adding a \+-tile to coordinate $(r,i-r+2)$ yields a reduced pipe dream $m_{s,\beta}D$, we have $i+1\notin\DES_L(\rho)$.

The conclusion does not depend on whether $(r,j_{p^-})\in\cross(\nabla D)$.
\vskip 0.5em

\noindent\textbf{Subcase (1)(a)(iii):} $(r,i-r+1)\in C_q$, for some $q<p$. By Lemma~\ref{lem:monk-increase-one}, $\pop(m_{s,\beta}D)=(i+1,r)$ since the first \+-tile of $m_{s,\beta}D$ must be located at $(r,i-r+2)$.

As above, the basic Monk steps for $m_{s,\beta}D$ and $m_{s,\beta}(\nabla D)$ agree for the steps up to $(q-1)^+$, where in $m_{s,\beta}D$, a double crossing is created at $(i_{q-1^+},j_{q-1^+})$ and at $(i_{q^-},j_{q^-})=(r,j_{q^-})$. Either $(r,j_{q^-})\notin\cross(\nabla D)$ or $(r,j_{q^-})\in\cross(\nabla D)$. Both situations will yield the same conclusion.

If $(r,j_{q^-})\notin\cross(\nabla D)$, i.e. $j_{q^-}=i-r+1$, then the process of $m_{s,\beta}$ applied to $\nabla D$ ends after step $(q-1)^+$. Since $m_{s,\beta}(\nabla D)$ would have a double crossing if $(r,i-r+1)$ were a \+-tile, $s_i\rho$ is not reduced and thus $i\in\DES_L(\rho)$. Now, $m_{s,\beta}(\nabla D)\in\BPD(\rho)$ is precisely the pipe dream of $m_{s,\beta}$ applied to $D$ after move $q^-$. The step $q^+$ of applying $m_{s,\beta}$  to $D$ adds a \+-tile at coordinate $(r,i-r+2)$ whose corresponding simple transposition is $s_{i+1}$, which creates a double crossing with the \+-tile at $(i_{q+1^-},j_{q+1^-})$, as $q<p$. This means that $i+1\in\DES_L(\rho)$. 

To simulate the steps $q+1^-,q+1^+,\ldots,p^-,p^+$ of $m_{s,\beta} D$, we can uncross the \+-tile at coordinate $(i_{q+1^-},j_{q+1^-})$, which is a crossing between pipe $i+1$ and $i+2$, of $m_{s,\beta} D$, and continue with step (2) of Definition~\ref{def:monk-pd}. This is saying that the basic Monk steps of $m_{\rho^{-1}(i+2),\rho^{-1}(i+1)}$ applying to $m_{s,\beta}(\nabla D)$ are precisely the same as the basic Monk steps $q+1^-,q+1^+,\ldots,p^-,p^+$ of $m_{s,\beta}$ applying to $D$, as the \+-tile at coordinate $(r,i-r+2)$ will not be involved in these steps by Lemma~\ref{lem:monk-increase-one}. See Figure~\ref{fig:pd-subcase1a3} for an example.
\begin{figure}[h!]
\centering
\includegraphics[scale=0.3]{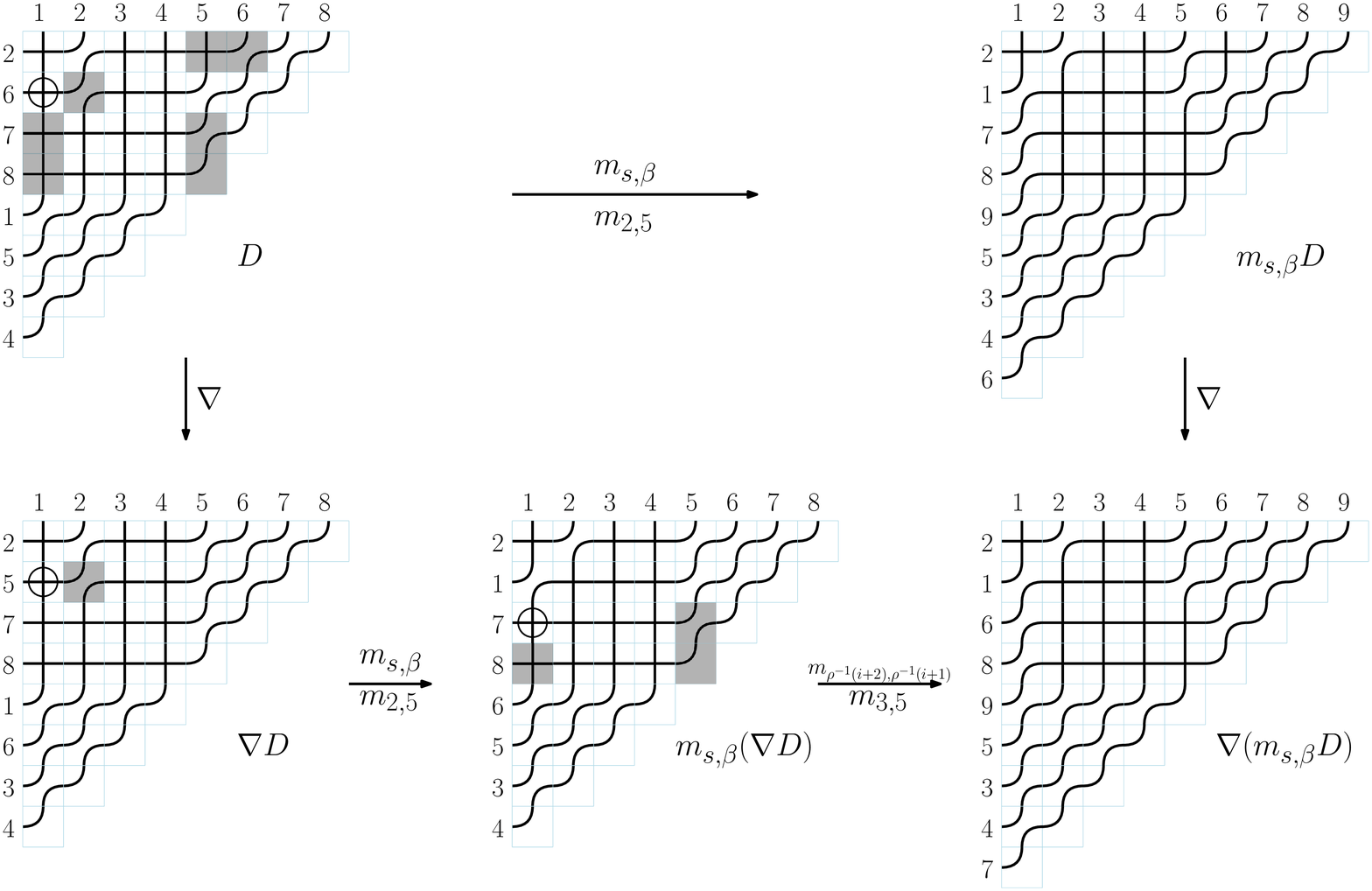}
\caption{An example for Subcase (1)(a)(iii) of Lemma~\ref{lem:PD-induction-main}, with $\pi=\rho=21786534$, $(i,r)=(5,1)$, $(s,\beta)=(2,5)$. The first step $1^-$ of each Monk move is circled and all the other grids in Monk's footsteps are shaded.}
\label{fig:pd-subcase1a3}
\end{figure}

If $(r,j_{q^-})\in\cross(\nabla D)$, i.e. $j_{q^-}<i-r+1$, then the the process of $m_{s,\beta}$ applied to $\nabla D$ ends after step $q^+$, which yields a pipe dream that equals $m_{s,\beta}$ applied to $D$ after step $q^-$ so $\rho=\perm(m_{s,\beta}(\nabla D))=\pi$. The first \+-tile of $m_{s,\beta}(\nabla D)$ is at coordinate $(r,i-r+1)$ so $i\in\DES_L(\rho)$, and since we would have a double crossing if the coordinate $(r,i-r+2)$ were a \+-tile, $i+1\in\DES_L(\rho)$. To simulate the steps $q+1^-,q+1^+,\ldots,p^-,p^+$ for computing $m_{s,\beta}$ applied to $D$, we can as above uncross the \+-tile at $(i_{q+1^-},j_{q+1^-})$ by applying $m_{\rho^{-1}(i+2),\rho^{-1}(i+1)}$, and concluding that $\nabla (m_{s,\beta}D)=m_{\rho^{-1}(i+2),\rho^{-1}(i+1)} m_{s,\beta}(\nabla D)$.

As a summary for Case (1)(a), we divide into subcases based on the local conditions of whether and where the first \+-tile $(r,i-r+1)$ of $D$ appears in the complete Monk footprints $C$, but it turns out that this division is also governed by the ``global condition" of whether $i\in \DES_L(\rho)$ and $i+1\in\DES_L(\rho)$ where $\rho=\perm(m_{s,\beta}(\nabla D))$. Specifically, if $i\notin\DES_L(\rho)$, we are in Subcase (1)(a)(i); if $i\in\DES_L(\rho)$ and $i+1\notin\DES_L(\rho)$, we are in Subcase (1)(a)(ii); if $i,i+1\in\DES_L(\rho)$, we are in subcase (1)(a)(iii). The analysis above within the subcases agree with the lemma statement, so we are done with Case (1)(a).
\vskip 0.5em

\noindent\textbf{Case (1)(b):} $(s,\beta)=(\pi^{-1}(i+1),\pi^{-1}(i))$ where $\pop(D)=(i,r)$. As in case (1)(a), let the Monk footprints be $F=\{(i_{1^-},j_{1^-}),\ldots,(i_{p^+},j_{p^+})\}$ and define $\rho:=\perm(\nabla D)$. The case condition says that $(i_{1^-},j_{1^-})=(r,i-r+1)$, $(i_{1^+},j_{1^+})=(r,i-r+2)$. By Lemma~\ref{lem:monk-increase-one}, the first \+-tile remains at coordinate $(r,i-r+2)$ throughout the rest of the Monk steps $2^-,\ldots,p^+$. Thus, we already have $\pop(m_{s,\beta} D)=(i+1,r)$.
\vskip 0.5em

\noindent\textbf{Subcase (1)(b)(i):} $p=1$. Here, the \+-tile at $(r,i-r+2)$, which corresponds to the simple transposition $s_{i+1}$, does not create a double crossing, so $i+1\notin \DES_L(\rho)$. As $m_{s,\beta}(D)$ is obtained from $D$ by moving the first \+-tile one step right, $\nabla (m_{s,\beta}D)=\nabla D$ as desired.
\vskip 0.5em

\noindent\textbf{Subcase (1)(b)(ii):} $p\geq2$. Here, the \+-tile at $(r,i-r+2)$ creates a double crossing after step $1^+$, so $i+1\in\DES_L(\rho)$. To simulate the basic Monk steps $2^-,2^+,\ldots,p^-,p^+$ of $m_{s,\beta}$ applied to $D$ on $\nabla D$, we need to remove the \+-tile at $(i_{2^-},j_{2^-})$ which is the crossing between pipe $i+1$ and $i+2$, and then continue applying Step (2) of Definition~\ref{def:monk-pd}. This means that $m_{\rho^{-1}(i+2),\rho^{-1}(i+2)}$ 
applied to $\nabla D$ gives us $m_{s,\beta} D$ without the first \+-tile at $(r,i-r+2)$. So $m_{\rho^{-1}(i+2),\rho^{-1}(i+2)}(\nabla D)=\nabla(m_{s,\beta}D)$ as desired.
\vskip 0.5em

We are now done with Case (1), which is the map $m_{s,\beta}$. Case (2) deals with the map $x_{\alpha}{\rightsquigarrow}$. The exact same argument from Case (1)(a) is applicable to Case (2)(a), by dividing into subcases based on whether the first \+-tile of $D$ appears in the Monk footprints of $m_{s,\beta}D$, analyzing whether $i\in\DES_L(\rho)$, $i+1\in\DES_L(\rho)$ where $\rho:=\perm(x_{\alpha}{\rightsquigarrow} (\nabla D))$, and simulating the Monk steps of $m_{s,\beta}D$ via $x_{\alpha}{\rightsquigarrow} (\nabla D)$ and potentially one more Monk move $m_{\rho^{-1}(i+2),\rho^{-1}(i+1)}$. So we will not repeat the details here. Case (2)(b) is a degenrate case where the newly inserted \+-tile at coordinate $(\alpha,1)$ has become the first \+-tile of $x_{\alpha}{\rightsquigarrow} D$. We evidently have $\pop(x_{\alpha}{\rightsquigarrow} D)=(\alpha,\alpha)$ and $\nabla (x_{\alpha}{\rightsquigarrow} D)=D$ as desired.
\end{proof}

\subsection{Induction on bumpless pipe dreams}

Before spelling out the technical details, we take a moment to explain our ideas intuitively. Our main idea is to perform two corresponding Monk's rule computations on the bumpless pipe dreams $D$ and $\nabla D$ in parallel and study how these two processes are related. Generically, by ``corresponding Monk's rule computations'' we mean applying $x_\alpha{\rightsquigarrow}$ or $m_{s,\beta}$ to both $D$ and $\nabla D$. Lemma \ref{lem:BPD-induction-main} spells out the special cases in detail. Each computation of Monk's rule consists of a sequence of basic Monk moves. We will see that in most cases, we may cut the sequence of moves applied to $D$ into $m$ subsequences, and also cut the sequence applied to $\nabla D$ as $m$ corresponding subsequences (two corresponding subsequences might have different number of basic moves), such that before the start of each pair of corresponding subsequences, the two (almost) bumpless pipe dreams are related by $\nabla$, and the first basic moves in the corresponding subsequences start at \emph{corresponding tiles} (which we define below). Therefore in most cases, applying Monk's rule commutes with applying $\nabla$. In the critical cases when they don't commute, we may still find corresponding sequences as described above until a basic Monk move affect the last column move rectangle in a certain way. The detailed analysis of the critical cases is in Lemma \ref{lem:BPD-induction-main}. The technical difficulties on bumpless pipe dreams, as compared to on pipe dreams, come from the fact that $\nabla$ on pipe dreams is simply removing a cross, whereas on bumpless pipe dreams the procedure is not local.

We first extend the definition of $\nabla$ (Definition~\ref{def:pop}) to almost bumpless pipe dreams by introducing the following column moves:
\begin{enumerate}
\item[(2')] In the column moves defined in Step (2) of Definition \ref{def:pop}, simultaneously replace exactly one \+-tile of the same two pipes in the input and output with a \bt-tile. This is illustrated in Figure \ref{fig:newcolmoves2}.
 
\item[(3')] To extend the column moves defined in Step (3) of Definition \ref{def:pop}, add the following column moves: \begin{enumerate}
         \item If $(x,y)$ is a marked tile and the pipe passing through $(x,y+1)$ is $y+1$, and the \bt-tile is in column $y$ below row $x$, follow the same rule as described in Step (3) of Definition \ref{def:pop};
         \item If $(x,y)$ is a marked tile, pipes $y$ and $y+1$ cross at $(x',y+1)$ for some $x'>x$, and there is a \bt-tile of pipes $y+1$ and $q$ with coordinate $(z,y+1)$ for some $x<z<x'$, define the column move by first replacing the \bt-tile by a \+-tile and following Step (3) of Definition 3.1, and then replace the \+-tile of $y$ and $q$ by a \bt-tile.
     \end{enumerate}
     \begin{figure}[h!]
    \centering
    \includegraphics[scale=0.6]{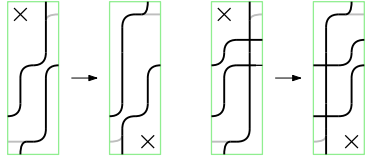}
    \caption{Extra column moves for almost bumpless pipe dreams in (2'). We suppress the four possibilities for each arrow by using gray  to indicate other possibilities.}
    \label{fig:newcolmoves2}
\end{figure}
      \begin{figure}[h!]
    \centering
    \includegraphics[scale=0.6]{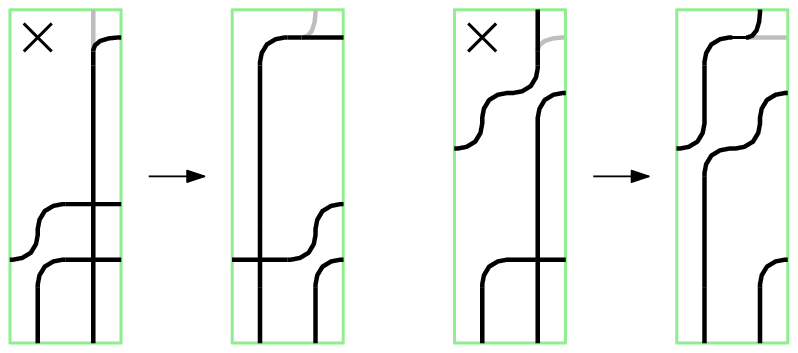}
    \caption{Extra column moves for almost bumpless pipe dreams in (3').}
    \label{fig:newcolmoves3}
\end{figure}
 \end{enumerate}

\begin{remark}
The operator $\nabla$ is not defined on all almost bumpless pipe dreams for $\pi$. See Figure \ref{fig:undefined} for examples. Specifically, there are the following cases. If $(x,y)$ is a marked tile, let $p$ be the pipe that passes through $(x,y+1)$ which is either a \vtile-tile or \rt-tile. Let $(x',y+1)$ be the position of the \jt-turn of $p$ in column $y+1$. If such a position does not exist, the column move is defined. If this is a \jt-tile, and the tile at $(x',y)$ is a \bt-tile of pipes $p$ and $q$, $q$ must also intersect $p$ at a tile between row $x$ and $x'$. In this case a column move is undefined. If $(x',y+1)$ is a \bt-tile, let $q$ be the other pipe of this \bt. If $p=y$ and $q=y+1$, the column move is undefined. Suppose now we can find $(x'',y+1)$ that is the \jt-tile of $q$ in column $y+1$. If $p$ and $q$ cross at $(x'',y)$, the column move is also not defined. It is okay to have these undefined cases, because these configurations will not appear right before the first moves of each pair of corresponding subsequences of basic Monk moves, in which case we need $\nabla$ to be defined.
\begin{figure}[h!]
    \centering
    \includegraphics[scale=0.6]{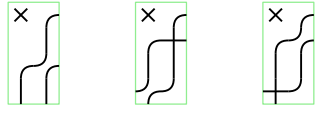}
    \caption{Examples for when a column move is not defined on an almost bumpless pipe dream (the bottom row of the first rectangle is also the bottom row of the bumpless pipe dream)}
    \label{fig:undefined}
\end{figure}
\end{remark}

Let $\D$ be an (almost) bumpless pipe dream of $\pi$. Suppose after this extension of the definition of $\nabla$, $\nabla$ is defined on $\D$.

\begin{defin}
Let $\E=\nabla \D$, and suppose $\pop(D)=(i,r)$. We say that a tile $(a,b)$ with $a\ge r$ in $\D$ and a tile $(a',b')$ with $a'\ge r$ in $\E$ are \emph{corresponding tiles} if in the case that $\D$ and $\E$ each contains a \bt-tile $(a,b)$ and $(a',b')$ are these \bt-tiles, or in the case that neither $\D$ nor $\E$ contains \bt-tiles, $(a,b)$ and $(a',b')$ are both \rt-tiles of pipes $\perm(D)(x)$ and $\perm(E)(x)$ for some $x$, and $a=a'$.  
\end{defin}

\begin{ex}
In Figure  \ref{fig:corresponding-tiles}, the first $(\D,\E)$ pair has pairs of corresponding tiles $((2,4),(2,3))$, $((3,3),(3,4))$, $((4,2),(4,2))$, $((4,6),(4,6))$, $((5,5),(5,4))$, and $((6,4),(6,5))$; the second $(\D,\E)$ pair has $((3,4),(4,3))$ as a pair of corresponding tiles.
\begin{figure}
    \centering
    \includegraphics[scale=0.6]{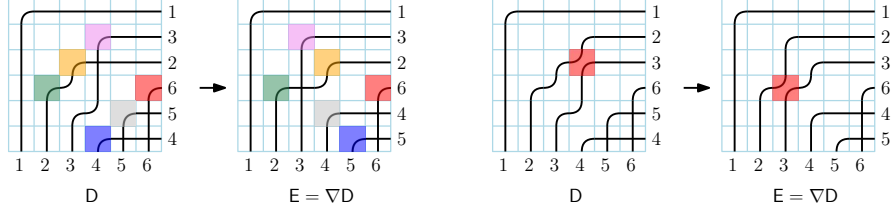}
    \caption{Examples of corresponding tiles}
    \label{fig:corresponding-tiles}
\end{figure}
\end{ex}

We remark that in the process of performing two corresponding Monk's rule computations on $D$ and $\nabla D$, at the beginning of each corresponding subsequences of basic moves, there will always be a single pair of corresponding tiles of interest that is determined by the algorithm, even though there can be multiple pairs of corresponding tiles in two bumpless pipe dreams related by $\nabla$ according to the definition. Furthermore, outside of the union 
$\mathcal{U}$ of column move rectangles of $\D$ and $\nabla \D$, a pair of corresponding tiles have the same coordinates, since $\D$ and $\nabla \D$ are identical outside of $\mathcal{U}$.


\begin{lemma}
\label{lem:BPD-commute-cases}
Suppose $\E=\nabla\D$, $\pop(\D)=(i,r)$, and $(a,b)$ in $\D$ and $(a',b')$ in $\E$ are corresponding tiles. Suppose $(a,b)$ is not the southernmost tile with an \rt-turn in column $i+1$.
Then there is a sequence of basic Monk moves following Steps (2) and (3) of Definition \ref{def:monk-bpd} starting with a $\mindroop$ at $(a,b)$ in $\D$, and a corresponding sequence (consisting of possibly different number of steps) of basic Monk moves starting at $(a',b')$ in $\E$ such that after applying both sequences of moves, the resulting (almost) bumpless pipe dreams $\widetilde{\D}$ and $\widetilde{E}$ satisfy  $\widetilde{\E} = \nabla \widetilde{\D}$ and $\pop(\widetilde{\D})=(i,r)$. Furthermore, $\widetilde{\D}$ has more basic Monk moves available if and only if $\widetilde{\E}$ does, and these moves are $\mindroop$s at corresponding tiles.
\end{lemma}

\begin{proof}
Let $\mathcal{U}$ be the union of column move rectangles for $\D$ and $\E$. If the basic Monk move at $(a,b)$ affects only tiles in outside of $\mathcal{U}$, then $(a,b) = (a',b')$ and we simultaneously perform the move at $(a,b)$ in $\D$ and $\E$ and get $\widetilde{\D}$ and $\widetilde{\E}$. The next basic Monk moves in $\widetilde{\D}$ and $\widetilde{\E}$ start at the same location.

If the $\mindroop$ at $(a,b)$ affects tiles in $\mathcal{U}$, let $U$ be the column move rectangle that contains $(a,b)$ if it exists; otherwise let $U$ be the column move rectangle that intersects row $a$ closest to $(a,b)$. Notice that the tile $(a,b)$ could be to the left or right of $U$.
We now do a detailed case analysis.  The readers are strongly recommended to refer to the figures while reading. In each accompanying figure of this proof, the vertical arrows are $\nabla$ and the horizontal arrows correspond to sequences of basic Monk moves described in the text; the positions of $\D$, $\E$, $\widetilde{\D}$, and $\widetilde{\E}$ are NW, SW, NE, and SE, respectively. The corresponding tiles in $\D$ and $\E$ where the sequence of basic Monk moves start at are circled.
\vskip 0.5em

\noindent\textbf{Case (1)} (Figure \ref{fig:case1}). If the tile $(a,b)$ is to the right of $U$, the only way a $\mindroop$ at $(a,b)$ could affect a column move rectangle is when $(a,b)$ is the \rt-turn of a pipe $p$   with a \vtile-tile at some $(x,b)$ with $x>a$, where $(x,b)$ is the NE corner of a column move rectangle $U'$.
We perform the $\mindroop$ at $(a,b)$.
\begin{enumerate}[(a)]
\item If this $\mindroop$ does not create a \bt-tile, let the result be $\widetilde{\D}$. 
The only affected tile in $U'$ by this move is $(x,b)$. The corresponding $\mindroop$ move at $(a, b)$ in $\E$ changes the tiles in exactly the same way outside of $U'$, and at $(x,b)$ the $\jt$-tile becomes a \htile-tile. Let $\widetilde{\E}$ be the result after this move. Then $\widetilde{\E}=\nabla\widetilde{\D}$, and the next basic Monk move in $\widetilde{\D}$ is at $(x,b)$, whereas the next move in $\widetilde{\E}$ is at $(x,b-1)$. These are corresponding tiles. 
\item If this $\mindroop$ creates a \bt-tile, we check if a $\cbswap$ is necessary. If not, the situation is similar to (a); otherwise, we perform a $\cbswap$ move and let the result be $\widetilde{\D}$. We also perform the $\mindroop$ at $(a,b)$ in $\E$ followed by a $\cbswap$ and obtain $\widetilde{\E}$. We have $\widetilde{\E}=\nabla \widetilde{\D}$, and the next $\mindroop$s in $\widetilde{\D}$ and $\widetilde{\E}$ are at the corresponding \bt- tiles.
\end{enumerate}
\begin{figure}[htp]
    \centering
    \includegraphics[scale=0.5]{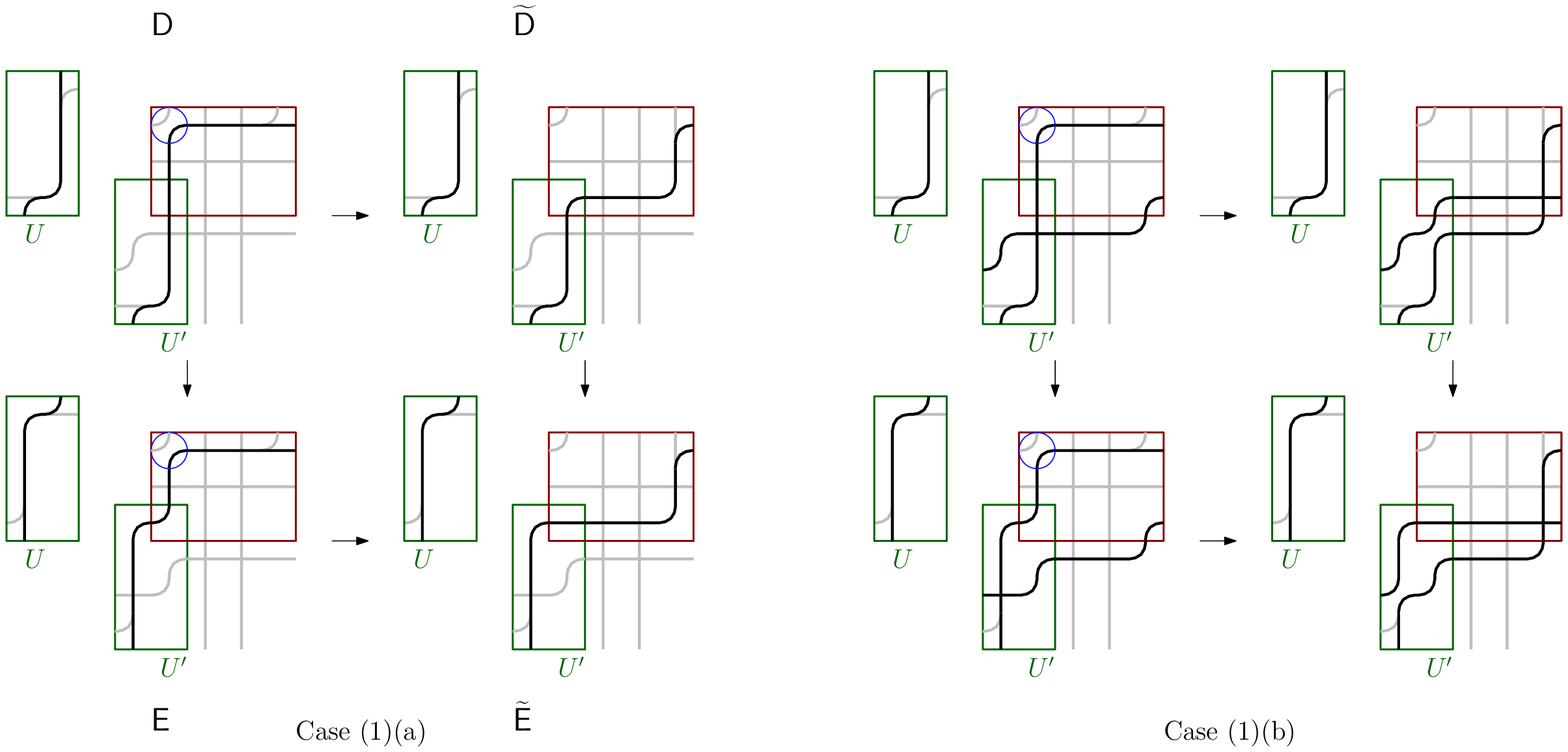}
    \caption{Case (1) of Lemma~\ref{lem:BPD-commute-cases}. In (1)(b) we show the case when the new \bt-tile is in $U'$, but it's possible that the new \bt-tile is outside of a column move rectangle. }
    \label{fig:case1}
\end{figure}

Now we consider Cases (2) to (4) when $(a,b)$ is inside a column move rectangle $U$, and $U$ is not the southernmost column move rectangle. Let $U'$ be the column move rectangle immediately SE of $U$. Let $(x,y)$ be the NE corner of $U$. $U'$ could be the southernmost column move rectangle, but the analysis for this case is similar to the case when $U'$ is not. Therefore we focus attention on the case when $U'$ is not the southernmost column move rectangle. Suppose the pipe that goes through the NE corner of $U$ is $p$ and the pipe that goes through the NE corner of $U'$ is $p'$.
    
    \vskip 0.5em
\noindent\textbf{Case (2)} (Figure \ref{fig:case2}).  The tile $(x,y)$ in $\D$ is an \rt-tile and $(a,b)=(x,y)$. In this case, $(a',b')=(x,y-1)$. There are three subcases to consider:
    \begin{enumerate}[(a)]
        \item $U$ and $U'$ have nonempty intersection, and the pipe $p'$ has an \rt-tile in column $y+1$ that is below row $x$.    First we argue that in this case the width of the rectangle that bounds this $\mindroop$ cannot be larger than 2, because otherwise, the tile $(x,y+1)$ must be a \+-tile, and the pipe that contains the vertical segment of this cross would be forced to cross $p$ twice.

         In $\D$, starting at $(x,y)$, perform all consecutive basic Monk moves that are available in column $y$. At the same time, perform all consecutive $\mindroop$ moves that are available in column $y-1$ in $\E$.
        Pipes $p$ and $p'$ would bump at a $\bt$-tile in $\D$ and $\E$. 
        If $p$ and $p'$ do not intersect, we replace this $\bt$-tile with a \+ tile in $\D$ and $\E$ and let the results be $\widetilde{\D}$ and $\widetilde{\E}$. 
        If $p$ and $p'$ already intersect, we simultaneously perform the $\cbswap$ move in $\D$ and $\E$, and let the results be $\widetilde{\D}$ and $\widetilde{\E}$. Since the column moves run NW to SE and the $\cbswap$ moves modify tiles that are NE-SW to each other, the new \bt-tiles created in $\D$ and $\E$ after the corresponding $\cbswap$ moves must be outside of $\mathcal{U}$, and therefore at the same location. 
        
        Suppose $(z,y+1)$ is the SE corner of $U'$. Then we see that the rectangle $U''$ with $(x,y)$ as the NW corner and $(z,y+1)$ as the SE corner is a column move rectangle. Replacing $U\cup U'$ with $U''$ in $\mathcal{U}$, we see that $\widetilde{\E}=\nabla \widetilde{\D}$.
        
        \item $U$ and $U'$ have nonempty intersection, $(x,y+1)$ is a \+-tile, and $p'$ has an \rt-tile in column $y+1$ of $\D$ that is  above row $x$. In this case, perform all the basic Monk moves available in column $y$ in $\D$. If a \bt-tile is created at the end, perform also the next move that replaces this tile with a \+. Let the result be $\widetilde{\D}$. In $\E$, perform all the basic Monk moves available in column $y-1$.  In this case, $p$ would create a bump with $p'$ at $(x',y)$, the SE corner of $U$. Since $p$ and $p'$ already cross at $(x,y+1)$, a $\cbswap$ needs to be performed, and after this we perform all the basic Monk moves available in column $y+1$. If a \bt-tile is created at the end, perform also the next move that replaces this tile with a \+. Let the result be $\widetilde{\E}$.  Suppose $(z,y+1)$ is the SE corner of $U'$. Let $U''$ be the rectangle with $(x,y)$ as its NW corner and $(z,y+1)$ as its SE corner. Then $U''$ is a column move rectangle. Replacing $U\cup U'$ with $U''$ in $\mathcal{U}$, we see that $\widetilde{\E} = \nabla \widetilde{\D}$. It is also easy to see that the next basic Monk moves in $\widetilde{\D}$ and $\widetilde{\E}$ are at corresponding tiles. 
        
        \item If $U$ and $U'$ do not intersect, we exhaust all the basic Monk moves available in column $y$ in $\D$ and let the result be $\widetilde{\D}$. Correspondingly, we exhaust all moves available in column $y-1$ in $\E$ and let the result be $\widetilde{\E}$.   Let $U''$ be the rectangle obtained by shifting $U$ to the right by one tile. 
        Replacing $\mathcal{U}$ with $(\mathcal{U}\setminus U)\cup U''$, we see that $\widetilde{\E}=\nabla \widetilde{\D}$. It is also easy to see that the next basic Monk moves in $\widetilde{\D}$ and $\widetilde{\E}$ are at corresponding tiles.
         
    \end{enumerate}
    \begin{figure}[htp]
        \centering
        \includegraphics[scale=0.5]{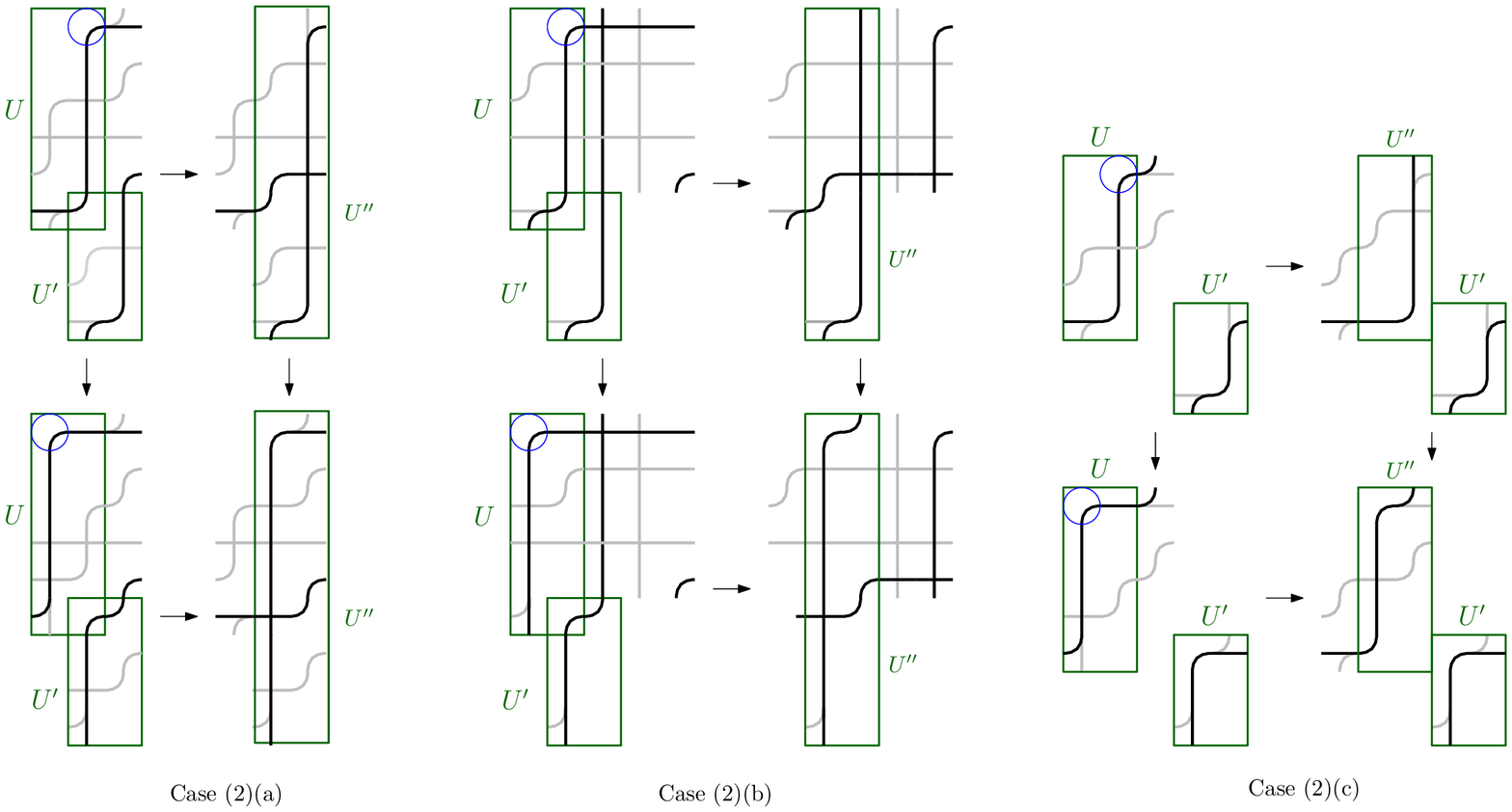}
        \caption{Cases (2) of Lemma~\ref{lem:BPD-commute-cases}. The illustration of (2)(b) omits the  case analysis similar to Case (1).}
        \label{fig:case2}
    \end{figure}

    \noindent\textbf{Case (3)} (Figure \ref{fig:case3}).
    The tile $(a,b)$ in $\D$ is a $\bt$ that previously was a \+ of $p$ and $q$ for some pipe $q$. Then $(a',b')$ is a also  $\bt$ tile that was previously a \+ of $p$ and $q$ in $\E$. The case here analysis is parallel to the three cases in (2), so we omit the details, but include the illustrations.
    \begin{figure}[htp]
        \centering
        \includegraphics[scale=0.5]{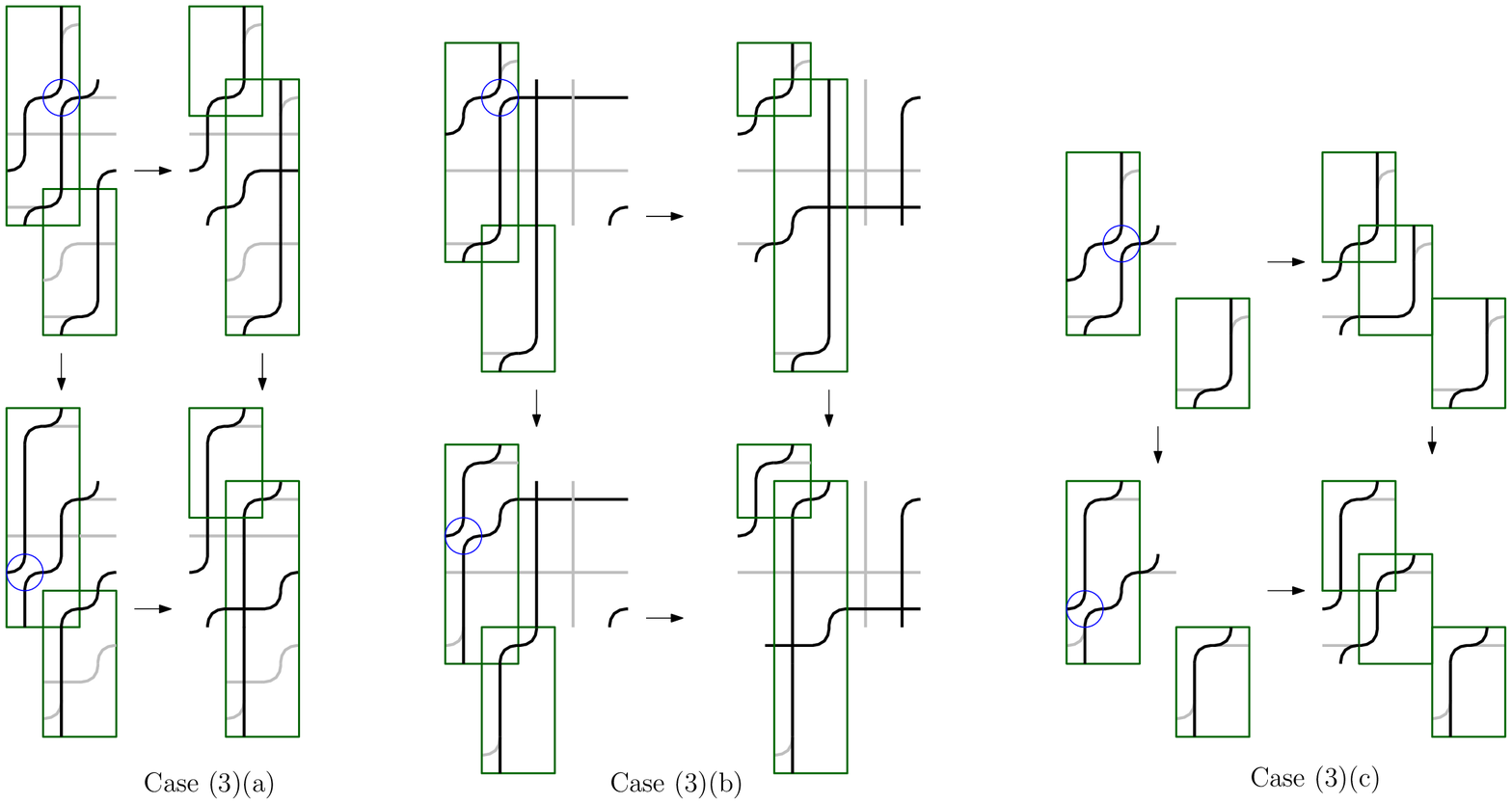}
        \caption{Cases (3) of Lemma~\ref{lem:BPD-commute-cases}.}
        \label{fig:case3}
    \end{figure}
    \vskip 0.5em

    \noindent\textbf{Case (4)} (Figure \ref{fig:case4}.) The \rt-turn at $(a,b)$ inside $U$ belongs to a pipe $q$ intersecting $p$ at $(a,y)$. Here $b=y-1$. In the figure, we show the cases for when $(a,b)$ is an \rt-tile or when $(a,b)$ is a \bt-tile separately. The corresponding $\mindroop$ in $\E$ is at $(a,y)$. In $\D$, exhaust the basic Monk moves in column $y-1$, and if at the end a \bt-tile is created, perform the move that replaces it with a \+-tile. In the meantime, exhaust the basic Monk moves in column $y$ of $\E$, and if at the end a \bt-tile is created, perform the move that replaces it with a \+-tile. These moves preserve the fact that $U$ is a column move rectangle. Let the results be $\widetilde{\D}$ and $\widetilde{\E}$. We have $\widetilde{\E}=\nabla \widetilde{\D}$. The next $\mindroop$s in $\widetilde{\D}$ and $\widetilde{\E}$, if exist, are at corresponding tiles. 
    \begin{figure}[htp]
        \centering
        \includegraphics[scale=0.5]{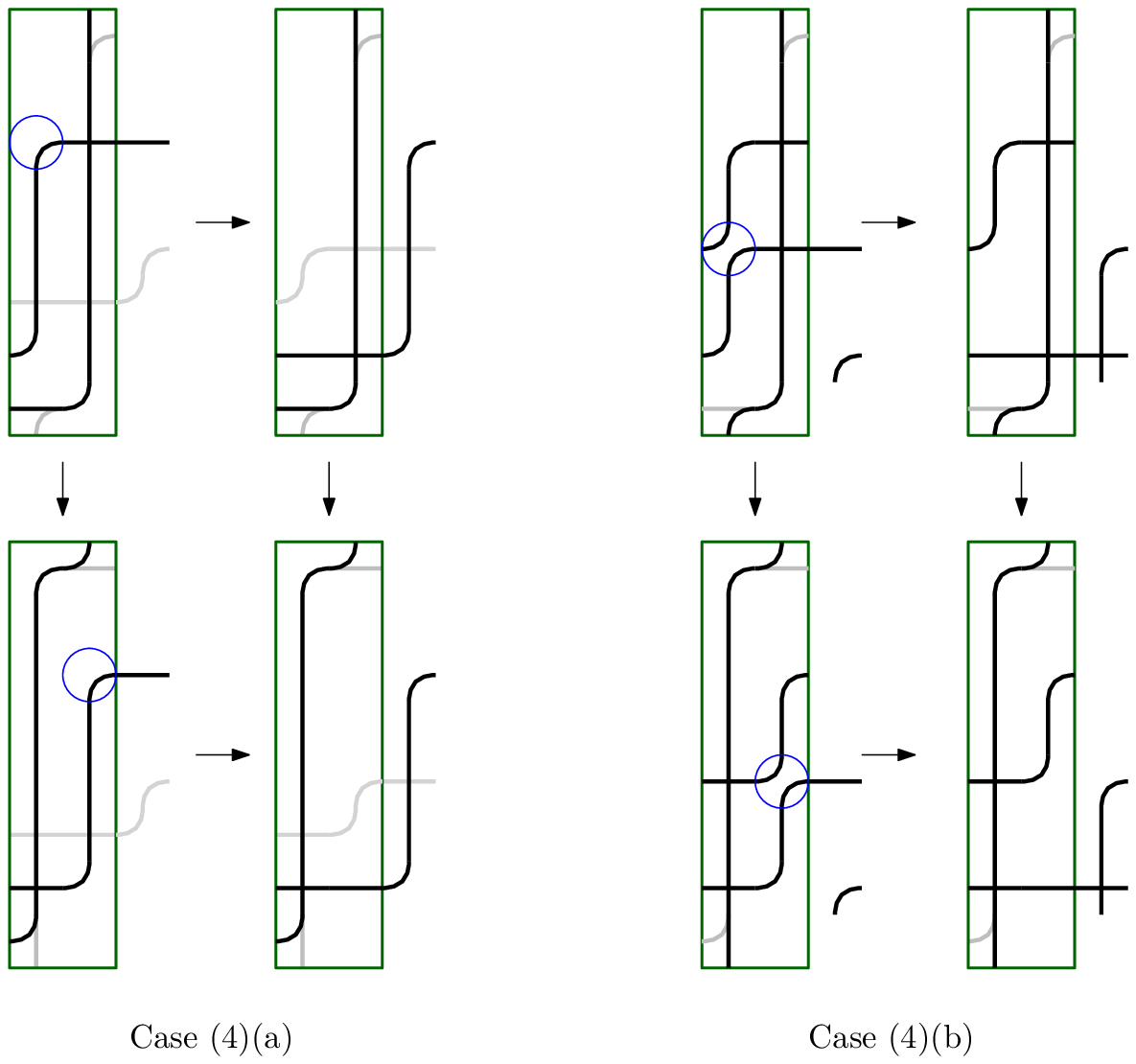}
        \caption{Case (4) of Lemma~\ref{lem:BPD-commute-cases}}
        \label{fig:case4}
    \end{figure}

    \vskip 0.5em
    
    \noindent We now consider the Cases (5) and (6) when $(a,b)$ is to the left of $U$, and $U$ is not the southernmost column move rectangle. 
    \vskip 0.5em
    \noindent\textbf{Case (5)} (Figure \ref{fig:case5}.)
    The SW corner $(x',y')$ of $U$ in $\D$ is a $\htile$-tile and $a=x'$, $b<y'$. We also have $(a',b')=(a,b)$. We perform a single $\mindroop$ at $(a,b)$ in both $\D$ and $\E$. Again, if a $\bt$-tile is created, we perform also the move that replaces it with a \+-tile. The results are $\widetilde{\D}$ and $\widetilde{\E}$.
    \begin{figure}[htp]
        \centering
        \includegraphics[scale=0.5]{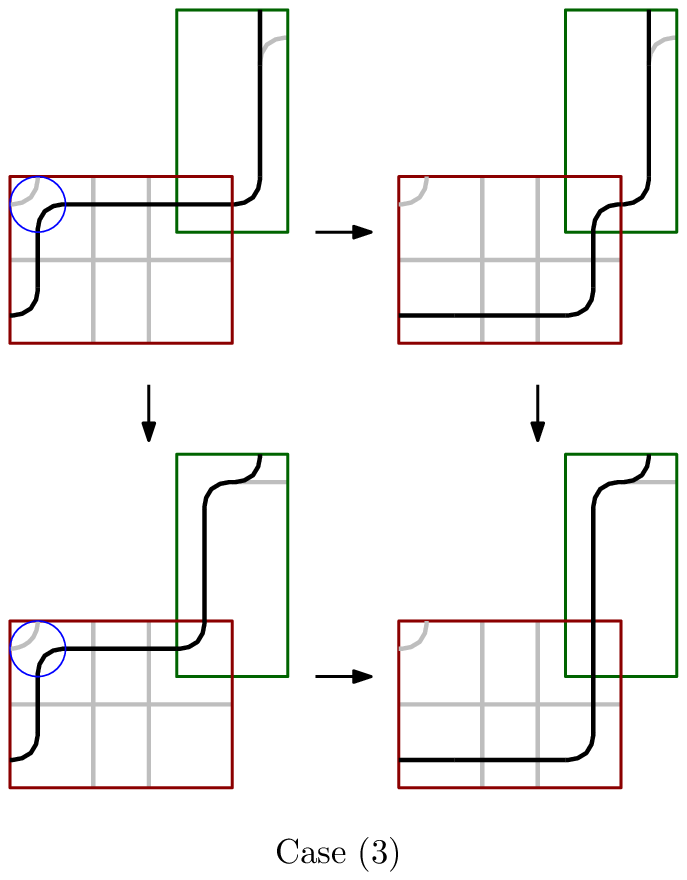}
        \caption{Case (5) of Lemma~\ref{lem:BPD-commute-cases}}
        \label{fig:case5}
    \end{figure}
    \vskip 0.5em
    
\noindent\textbf{Case (6)} (Figure \ref{fig:case6}).
The \rt-turn at $(a,b)$  belongs to a pipe $q$ intersecting $p$ at $(a,y)$, and $(a,b)$ is outside of $U$. Pipe $q$ intersects $p$ at $(a,y-1)$ in $\E$ and the corresponding $\mindroop$ in $\E$ is also at $(a,b)$. We exhaust all basic Monk moves in column $b$ in $\D$, and at the same time exhaust all basic Monk moves in column $b$ in $\E$. If this does not create a \bt-tile with pipe $p$, let the results be $\widetilde{\D}$ and $\widetilde{\E}$. $U$ remains a column move rectangle and $\widetilde{\E}=\nabla \widetilde{\D}$. This is illustrated as Case (6)(a). If a \bt-tile of $p$ and $q$ is created, we have again three cases to consider parallel to Case  (3). these are illustrated as Case (6), (b)--(d).
\vskip 0.5em
\begin{figure}[h!]
    \centering
    \includegraphics[scale=0.5]{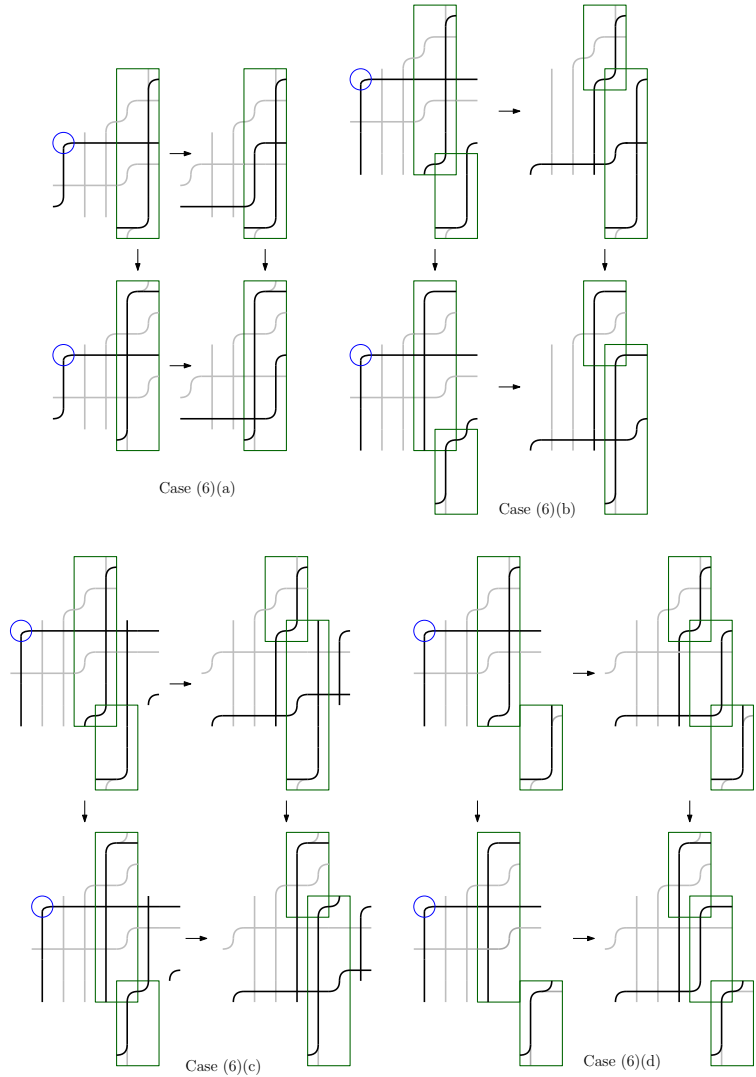}
    \caption{Case (6) of Lemma~\ref{lem:BPD-commute-cases}}
    \label{fig:case6}
\end{figure}
\noindent Now consider Cases (7) and (8) when $U$ is the southernmost column move rectangle. Recall from the assumption of the lemma that $(a,b)$ is not the southernmost \rt-turn in the east-side column of $U$. 

\vskip 0.5em

    \noindent\textbf{Case (7)} (Figure \ref{fig:case7}). Suppose $b=i$ and $(a,b)$ is the southernmost tile in column $i$ that contains an \rt-turn. Then $(a',b')=(a,i+1)$.
    We simultaneously exhaust the basic Monk moves in column $i$ of $\D$ and in column $i+1$ of $\E$. The last move in  either sequence must be a $\mindroop$ that creates a bump at $(c,j)$, the SE corner of the both $\mindroop$ moves. The next move replaces this bump with a \+-tile, and let the results be $\widetilde{\D}$ and $\widetilde{\E}$. We give separate illustrations for $(a,b)$ is an \rt-tile versus a \bt-tile.
    
    \begin{figure}
        \centering
        \includegraphics[scale=0.5]{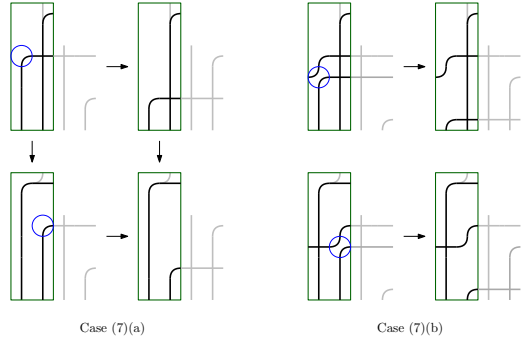}
        \caption{Case (7) of Lemma~\ref{lem:BPD-commute-cases}}
        \label{fig:case7}
    \end{figure}
    \vskip 0.5em
    
   \noindent\textbf{Case (8)} Suppose $b=i$ and $(a,b)$ is above the southernmost \rt-tile in column $i$. This is similar to Case (4).
   \vskip 0.5em
 
 \noindent Finally, we consider the cases when $U$ is the southernmost column move rectangle and $(a,b)$ is left of $U$.
 
 \vskip 0.5em
 \noindent 
    \noindent\textbf{Case (9)} (Figure \ref{fig:case9}.) These cases are similar to (6)(a). We omit the details, but include the illustration for reference.
    \begin{figure}[htp]
        \centering
        \includegraphics[scale=0.5]{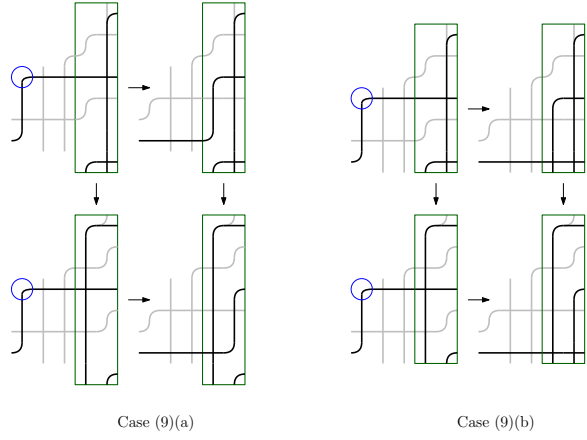}
        \caption{Case (9) of Lemma~\ref{lem:BPD-commute-cases}}
        \label{fig:case9}
    \end{figure}
\end{proof}

\begin{lemma}\label{lem:BPD-induction-main}
After replacing $\PD$'s by $\BPD$'s, the statements in Lemma~\ref{lem:PD-induction-main} hold.
\end{lemma}

\begin{proof} We give a proof for Case (1); Case (2) is similar.
\vskip 0.5em

\noindent\textbf{Case (1)(a).} 
Let $\D$ be the almost bumpless pipe dream of $\pi$ obtained from $D$ by replacing the crossing of pipes $\pi t_{s,\beta}(s)$ and $\pi t_{s,\beta}(\beta)$  with a \bt-tile, and let $\E=\nabla \D$. $\E$ is the almost bumpless pipe dream of $s_i\pi$ obtained from $\nabla D$ by replacing the crossing of pipes $s_i\pi t_{s,\beta}(s)$ and $s_i\pi t_{s,\beta}(\beta)$ in with a \bt-tile. 

Consider simultaneously the process of computing $m_{s,\beta}(D)$ and $m_{s,\beta}(\nabla D)$ by modifying $\D$ and $\E$ with basic Monk moves, and let $\D_0:=\D$ and $\E_0:=\E$. Applying Lemma \ref{lem:BPD-commute-cases} repeatedly, for each timestep $d=0,1,2,\cdots$, we let $\D_{d+1}:=\widetilde{\D_d}$ and $\E_{d+1}:=\widetilde{\E_d}$. Let $\mathcal{U}_d$ be the union of column move rectangles where $\D_d$ and $\E_d$ differ for each $d$, and let $U_d$ denote the southernmost one in each $\mathcal{U}_t$. We stop at step $t$ when either we encounter a $\mindroop$ at the southernmost \rt-turn $(a,i+1)$ in column $i+1$ of $\D$, in which case the corresponding $\mindroop$ is at the southernmost \rt-turn $(a',i)$ in column $i$ of $\E$, or such moves never occur and the computation of Monk's rule is complete. In the latter case, $\nabla(m_{s,\beta}(D))=m_{s,\beta}(\nabla D)$, $\pop(m_{s,\beta}(D))=(i,r)$, and $i\not\in \DES_L(\rho)$. We now do a careful analysis of the former case. 

We consider the case when 
$(a,i+1)$ in $\D_t$ is an \rt-tile where the case for when it is a \bt-tile is similar. In this case $a'=a$. We exhaust the basic Monk moves in column $i$ of $\E_t$ starting at $(a,i)$. When this is done, pipes $i$ and $i+1$ create a \bt-tile. Since these two pipes have not crossed before, replacing this $\bt$-tile with a \+-tile terminates the algorithm. Let the result be $\E_{t+1}$, so $\E_{t+1}=m_{s,\beta}(\nabla(\D))$. Notice that in this case $\rho=\pi$, and in particular $i\in \DES_L(\rho)$. 

If $i+1\not\in\DES_L(\pi)$, it must be the case that the southernmost \rt-tile in $\D_t$ of column $i+2$ is below row $a$, and pipes $i+1$ and $i+2$ do not cross. Therefore, after exhausting the basic Monk moves in column $i+1$ starting at $(a,i+1)$, pipes $i+1$ and $i+2$ create a \bt-tile, and we are done after we replace it with a \+-tile. Let the result be $\D_{t+1}$. The southernmost column move rectangle in $\D_{t+1}$ is $U_t$ shifted to the right by one tile. Therefore, $\nabla(m_{s,\beta}(D))=m_{s,\beta}(\nabla D)$, $\pop(m_{s,\beta}(D))=(i+1,r)$.

If $i+1\in \DES_L(\pi)$, we consider two cases:
\begin{enumerate}[(i)]
    \item The tile $(a,i+2)$ in $\D_t$ is a \+-tile where pipes $i+1$ and $i+2$ cross. Let $\D'_{t+1}$ be the computed from $\D_t$ by exhausting all basic Monk moves in column $i+1$. As a result, the southernmost column move rectangle in $\D'_{t+1}$ is $U_t$ shifted to the right by one tile. Meanwhile, in $\E_{t+1}$, the \+-tile of $i+1$ and $i+2$ is at $(a,i+2)$. Replacing it with a \bt-tile and exhausting the basic Monk moves in column $i+2$ gives $\E'_{i+1}$ so that $\nabla \D'_{t+1}=\E'_{i+1}$. We see that the rest of basic Monk moves that finish the computation of $m_{s,\beta}(D)$ and $m_{\pi^{-1}(i+2),\pi^{-1}(i+1)}(m_{s,\beta}(\nabla D))$ do not affect the union of column move rectangles. This case is illustrated in Figure \ref{fig:bpd1a_generic1}.

\begin{figure}[h!]
    \centering
    \includegraphics[scale=0.49]{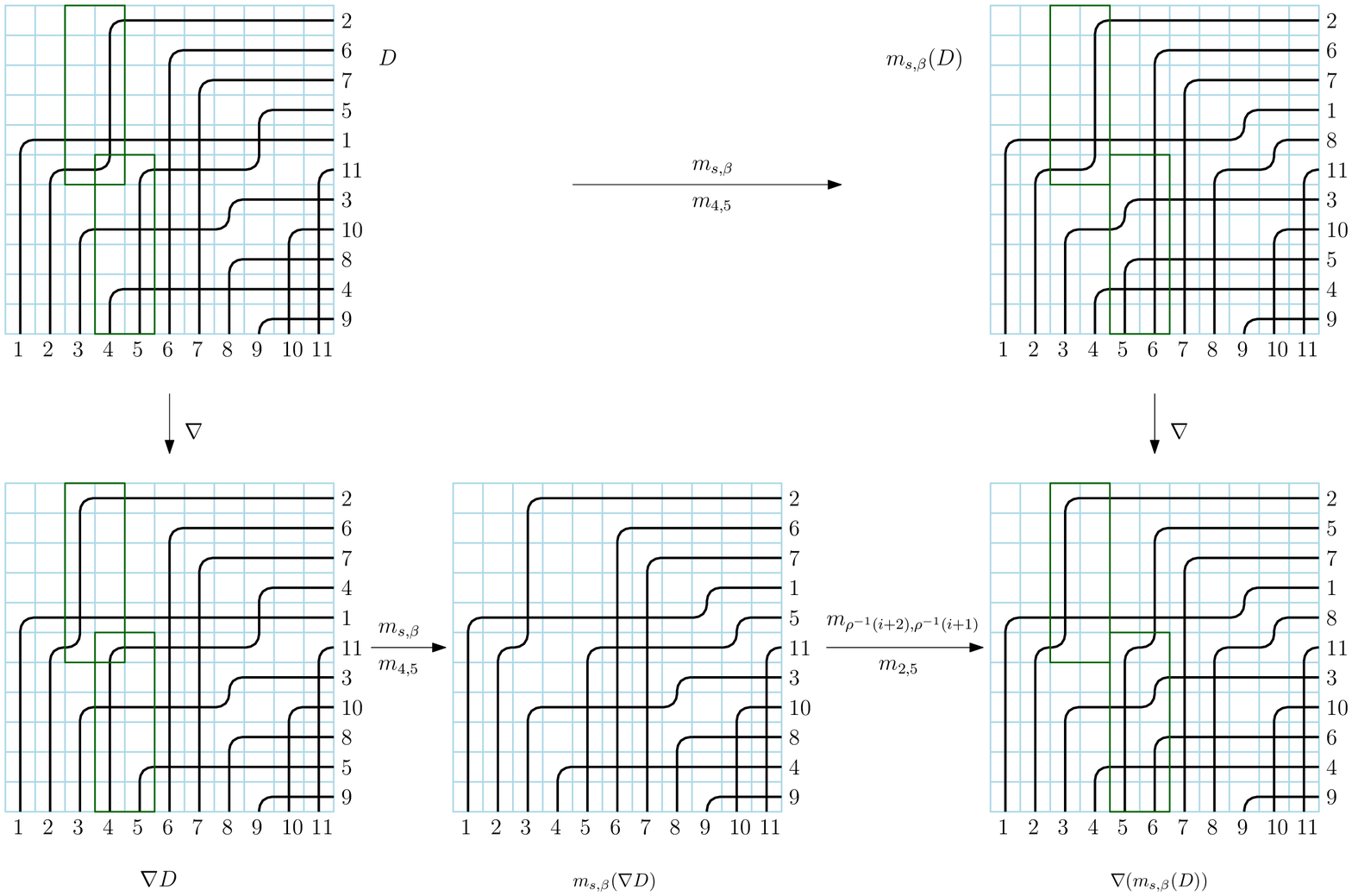}
    \caption{Case (1)(a)(i) of Lemma \ref{lem:BPD-induction-main}. }
    \label{fig:bpd1a_generic1}
\end{figure}

    \item In $\D_t$, pipes $i+1$ and $i+2$ cross in a column $j>i+2$. Let $\D'_{t+1}$ be obtained from $\D_t$ by exhausting the basic Monk moves in column $i+1$, followed by a $\cbswap$. Let $U_{t+1}$ be $U_t$ shifted to the right by one unit.
    At this point, $\D'_{t+1}$ and $\E_{t+1}$ differ within $U_{t+1}$ by a column move. In $\E_{t+1}$ replace the \+-tile of $i+1$ and $i+2$ in with a \bt-tile and let the result be $\E'_{t+1}$. Then $\nabla \D'_{t+1}=\E'_{t+1}$. Again, the remaining basic Monk moves only affect tiles outside of the union of column move rectangles. This case is illustrated in Figure \ref{fig:bpd1a_generic2}.
    
    \begin{figure}[h!]
    \centering
    \includegraphics[scale=0.49]{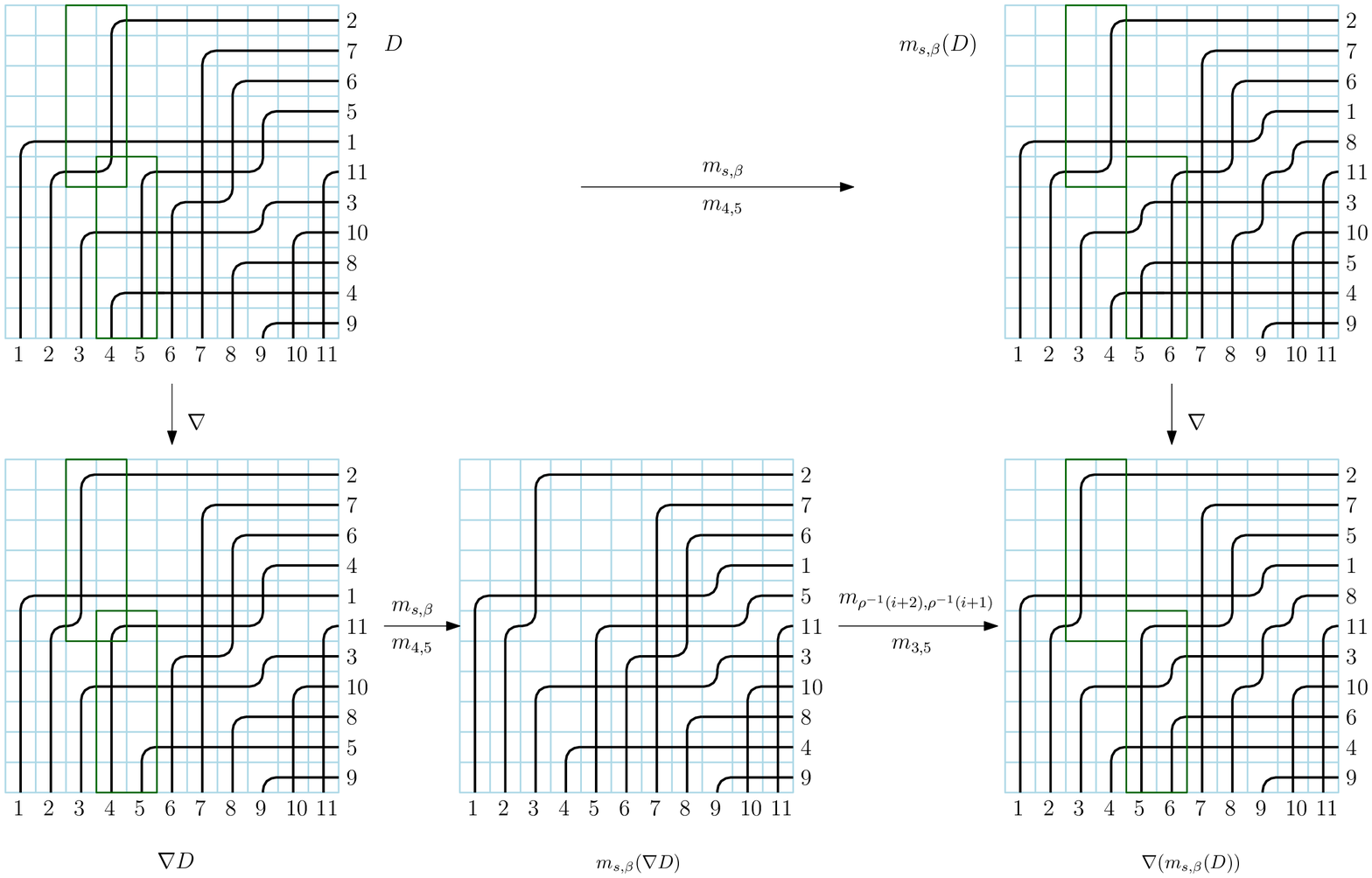}
    \caption{Case (1)(a)(ii) of Lemma \ref{lem:BPD-commute-cases}}.
    \label{fig:bpd1a_generic2}
\end{figure}
    
\end{enumerate}
In both cases, $\nabla(m_{s,\beta}(D))=m_{\rho^{-1}(i+2),\rho^{-1}(i+1)}(m_{s,\beta}(\nabla D))$ and $\pop(m_{s,\beta(D)})=(i+1,r)$.
\vskip 0.5em

\noindent\textbf{Case (1)(b).}  Let $\D$ be the almost bumpless pipe dream of $\pi$ obtained by replacing the \+-tile at $(a, i+1)$ of pipes $i$ and $i+1$ with a \bt-tile. In this case $\perm(\D)=s_i\pi=\rho$. 

If $i+1\not\in \DES_L(\rho)$, then it must be the case that in $\D$, there is an \rt-tile $(b,i+2)$ of pipe $i+2$  for some $b>a$, and $i+2$ does not cross with $i+1$. Therefore, exhausting the basic Monk moves in column $i+1$ of $\D$ and then replace the \bt-tile at $(b,i+2)$ with a \+-tile completes the computation $m_{s,\beta}(D)$. We see that in this case $\pop(m_{s,\beta}(D))=(i+1,r)$ and $\nabla(m_{s,\beta}(D))=\nabla D$. See Figure \ref{fig:bpd1a_special} for an illustration of this case. 

\begin{figure}[h!]
    \centering
    \includegraphics[scale=0.45]{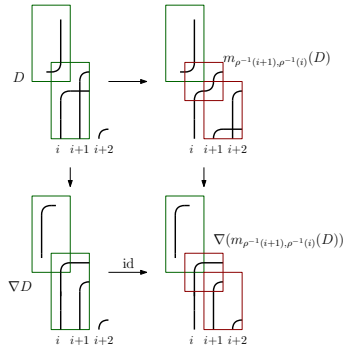}
    \caption{Case (1)(b) of Lemma \ref{lem:BPD-induction-main}.}
    \label{fig:bpd1a_special}
\end{figure}

If $i+1\in \DES_L(\rho)$, we may again consider two cases depending on whether the \+-tile of pipes $i+1$ and $i+2$ in $\D$ is in column $i+2$. The argument is similar to the parallel cases in (1)(a), so we omit the details.
\end{proof}

\section*{Acknowledgements}
We thank Allen Knutson and Alex Yong for helpful conversations.

\vskip 3em

\bibliographystyle{plain}
\bibliography{main}

\begin{thebibliography}{1}

\bibitem{bergeron1993rcgraph}
Nantel Bergeron and Sara Billey.
\newblock R{C}-graphs and {S}chubert polynomials.
\newblock {\em Experiment. Math.}, 2(4):257--269, 1993.

\bibitem{billey2019bijective}
Sara~C. Billey, Alexander~E. Holroyd, and Benjamin~J. Young.
\newblock A bijective proof of {M}acdonald's reduced word formula.
\newblock {\em Algebr. Comb.}, 2(2):217--248, 2019.

\bibitem{BJS}
Sara~C. Billey, William Jockusch, and Richard~P. Stanley.
\newblock Some combinatorial properties of {S}chubert polynomials.
\newblock {\em J. Algebraic Combin.}, 2(4):345--374, 1993.

\bibitem{fink2021zero-one}
Alex Fink, Karola M\'{e}sz\'{a}ros, and Avery St.~Dizier.
\newblock Zero-one {S}chubert polynomials.
\newblock {\em Math. Z.}, 297(3-4):1023--1042, 2021.

\bibitem{huang2020bijective}
Daoji Huang.
\newblock Bijective proofs of monk's rule for schubert and double schubert
  polynomials with bumpless pipe dreams.
\newblock {\em arXiv preprint arXiiv:2010.15048}, 2020.

\bibitem{huang2021schubert}
Daoji Huang.
\newblock Schubert products for permutations with separated descents.
\newblock {\em arXiv preprint arXiv:2105.01591}, 2021.

\bibitem{knutson2019schubert}
Allen Knutson.
\newblock Schubert polynomials, pipe dreams, equivariant classes, and a
  co-transition formula.
\newblock {\em arXiv preprint arXiv:1909.13777}, 2019.

\bibitem{LLS}
Thomas Lam, Seung~Jin Lee, and Mark Shimozono.
\newblock Back stable schubert calculus.
\newblock {\em arXiv preprint arXiv:1806.11233}, 2018.

\bibitem{lascoux1982polynomes}
Alain Lascoux and Marcel-Paul Sch\"{u}tzenberger.
\newblock Polyn\^{o}mes de {S}chubert.
\newblock {\em C. R. Acad. Sci. Paris S\'{e}r. I Math.}, 294(13):447--450,
  1982.

\end{thebibliography}
\end{document}